\documentclass[11pt,reqno,tbtags,draft]{amsart}
\usepackage{amssymb}
\usepackage{etoolbox} 
\usepackage{url}
\usepackage[square,numbers]{natbib}
\bibpunct[, ]{[}{]}{;}{n}{,}{,}

\title[Shape functional; sum of powers of subtree sizes and its mean]
{Conditioned Galton--Watson trees:\\ 
The shape functional, and more on the sum of powers of subtree sizes
and its mean}

\date{January 23, 2023}

\newcommand\urladdrx[1]{{\urladdr{\def~{{\tiny$\sim$}}#1}}}
\author{James Allen Fill}
\address{Department of Applied Mathematics and Statistics,
The Johns Hopkins University,
3400 N.~Charles Street,
Baltimore, MD 21218-2682 USA}
\email{jimfill@jhu.edu}
\urladdrx{http://www.ams.jhu.edu/~fill/}
\thanks{Research of the first author supported 
by the Acheson~J.~Duncan Fund for the Advancement of Research in
Statistics.}

\author{Svante Janson}
\thanks{Research of the second and third authors supported by the Knut and Alice
  Wallenberg Foundation} 
\address{Department of Mathematics, Uppsala University, PO Box 480,
SE-751~06 Uppsala, Sweden}
\email{svante.janson@math.uu.se}
\urladdrx{http://www2.math.uu.se/~svante/}

\author{Stephan Wagner}
\address{Department of Mathematics, Uppsala University, PO Box 480,
SE-751~06 Uppsala, Sweden}
\email{stephan.wagner@math.uu.se}

\keywords{Conditioned Galton--Watson tree, simply generated random tree, additive functional, tree recurrence, subtree sizes, shape functional, generating function, singularity analysis, Hadamard product of power series, method of moments, polylogarithm, Laplace transform}
\subjclass[2020]{Primary:\ 05C05; Secondary:\ 60F05, 60C05}

\numberwithin{equation}{section}

\renewcommand\leq{\le}
\renewcommand\geq{\ge}
\renewcommand\le{\leqslant}
\renewcommand\ge{\geqslant}

\allowdisplaybreaks

\theoremstyle{plain}
\newtheorem{theorem}{Theorem}[section]
\newtheorem{lemma}[theorem]{Lemma}

\theoremstyle{definition}
\newtheorem{example}[theorem]{Example}
\newtheorem{definition}[theorem]{Definition}
\newtheorem{problem}[theorem]{Problem}
\newtheorem{remark}[theorem]{Remark}

\AtEndEnvironment{remark}{\null\hfill\qedsymbol}
\AtEndEnvironment{example}{\null\hfill\qedsymbol}

\newenvironment{romenumerate}[1][-10pt]{
\addtolength{\leftmargini}{#1}\begin{enumerate}
 }{\end{enumerate}}

\newcounter{oldenumi}
{\setcounter{oldenumi}{\value{enumi}}
\begin{romenumerate} \setcounter{enumi}{\value{oldenumi}}}
{\end{romenumerate}}

\newcounter{thmenumerate}
\newenvironment{thmenumerate}
{\setcounter{thmenumerate}{0}%
 \def\item{\par
 \refstepcounter{thmenumerate}\textup{(\roman{thmenumerate})\enspace}}
}
{}

\newcounter{xenumerate}   

\newcommand{\refT}[1]{Theorem~\ref{#1}}
\newcommand{\refTs}[1]{Theorems~\ref{#1}}

\newcommand{\refL}[1]{Lemma~\ref{#1}}

\newcommand{\refR}[1]{Remark~\ref{#1}}
\newcommand{\refS}[1]{Section~\ref{#1}}

\newcommand{\refSS}[1]{Section~\ref{#1}}

\newcommand{\refD}[1]{Definition~\ref{#1}}

\newcommand{\refApp}[1]{Appendix~\ref{#1}}

\begingroup
  \divide\count255 by 60
  \multiply\count255 by -60
  \advance\count255 by \time
  \ifnum \count255 < 10 \xdef\klockan{\the\count1.0\the\count255}
  \else\xdef\klockan{\the\count1.\the\count255}\fi
\endgroup

\newcommand\noqed{\renewcommand{\qed}{}} 

\DeclareMathOperator*{\sumsum}{\sum\sum}

\DeclareMathOperator*{\sumxx}{\sum\nolimits^{**}}
\DeclareMathOperator*{\sumo}{\sum\nolimits^{\circ}}

\newcommand{\sumko}{\sum_{k=0}^\infty}

\newcommand{\summl}{\sum_{m=0}^\ell}
\newcommand{\sumno}{\sum_{n=0}^\infty}
\newcommand{\sumn}{\sum_{n=1}^\infty}

\newcommand{\sumim}{\sum_{i=1}^m}

\newcommand\set[1]{\ensuremath{\{#1\}}}

\newcommand\xpar[1]{(#1)}
\newcommand\bigpar[1]{\bigl(#1\bigr)}
\newcommand\Bigpar[1]{\Bigl(#1\Bigr)}

\newcommand\lrpar[1]{\left(#1\right)}
\newcommand\bigsqpar[1]{\bigl[#1\bigr]}

\def\rompar(#1){\textup(#1\textup)}    
\newcommand\xfrac[2]{#1/#2}

\def\xexp(#1){e^{#1}}

\newcommand\floor[1]{\lfloor#1\rfloor}

\newcommand\ntoo{\ensuremath{{n\to\infty}}}

\newcommand\mtoo{\ensuremath{{m\to\infty}}}

\newcommand\downto{\searrow}
\newcommand\upto{\nearrow}
\newcommand\half{\tfrac12}

\newcommand\punkt{.\spacefactor=1000}    
    
\newcommand\ie{i.e\punkt}
\newcommand\eg{e.g\punkt}

\newcommand\cf{cf\punkt}
\newcommand{\as}{a.s\punkt}

\newcommand\ii{\mathrm{i}}

\newcommand{\tend}{\longrightarrow}
\newcommand\dto{\overset{\mathrm{d}}{\tend}}
\newcommand\pto{\overset{\mathrm{p}}{\tend}}

\newcommand\bbR{\mathbb R}
\newcommand\bbC{\mathbb C}

\newcounter{CC}
\newcommand{\CC}{\stepcounter{CC}\CCx} 
\newcommand{\CCx}{C_{\arabic{CC}}}     
\newcounter{cc}
\newcommand{\cc}{\stepcounter{cc}\ccx} 
\newcommand{\ccx}{c_{\arabic{cc}}}     
\newcommand{\ccdef}[1]{\xdef#1{\ccx}}     
\newcommand{\ccname}[1]{\cc\ccdef{#1}}    
\newcommand{\ccreset}{\setcounter{cc}0} 

\renewcommand\Re{\operatorname{Re}}
\renewcommand\Im{\operatorname{Im}}

\newcommand\E{\operatorname{\mathbb E{}}}
\renewcommand\P{\operatorname{\mathbb P{}}}
\newcommand\Var{\operatorname{Var}}

\newcommand\Po{\operatorname{Po}}
\newcommand\Bi{\operatorname{Bi}}

\newcommand\Ge{\operatorname{Ge}}

\newcommand\sgn{\operatorname{sgn}}

\newcommand\ga{\alpha}
\newcommand\gb{\beta}
\newcommand\gd{\delta}
\newcommand\gD{\Delta}

\newcommand\gam{\gamma}
\newcommand\gG{\Gamma}

\newcommand\gl{\lambda}

\newcommand\gs{\sigma}
\newcommand\gss{\sigma^2}

\newcommand\gth{\theta}
\newcommand\eps{\varepsilon}

\newcommand\cP{\mathcal P}

\newcommand\cT{{\mathcal T}}

\newcommand\smatrixx[1]{\left(\begin{smallmatrix}#1\end{smallmatrix}\right)}

\newcommand\qw{^{-1}}
\newcommand\qww{^{-2}}
\newcommand\qq{^{1/2}}
\newcommand\qqw{^{-1/2}}

\newcommand\oi{[0,1]}
\newcommand\ooo{[0,\infty)}

\newcommand\dd{\,\mathrm{d}}

\newcommand{\pgf}{probability generating function}
\newcommand{\mgf}{moment generating function}

\newcommand\lhs{left-hand side}
\newcommand\rhs{right-hand side}

\newcommand\GW{Galton--Watson}
\newcommand\GWt{\GW{} tree}
\newcommand\cGWt{conditioned \GW{} tree}

\newcommand\tX{{\widetilde X}}
\newcommand\tY{{\widetilde Y}}
\newcommand\tPhi{{\widetilde \Phi}}

\newcommand\tn{\cT_n}
\newcommand\tnv{\cT_{n,v}}

\newcommand\rea{\Re\ga}
\newcommand\rga{{\Re\ga}}

\newcommand\hX{\widehat X}

\newcommand\gdd{\frac{\gd}{2}}

\newcommand\txi{\tilde\xi}
\newcommand\tmu{\tilde\mu}

\newcommand\Li{\operatorname{Li}}

\newcommand\tR{\tilde R}

\newcommand\doi{D_{1}}

\newcommand\GDD{$\gD$-domain}

\newcommand\gda{$\gD$-analytic}

\newcommand\xxm{^{(m)}}

\newcommand\ctn{\cT_n}

\newcommand\Oz[1]{O\bigpar{|1-z|^{#1}}}

\newcommand\kk{\varkappa}
\newcommand\hkk{\widehat\kk}
\newcommand\kkx{\kk^*}
\newcommand\kky{\kk^*}
\newcommand\kkk{\kappa}
\newcommand\kkkx{\kkk^*}
\newcommand\hkkk{\widehat\kkk}
\newcommand\yz{\bigpar{y(z)}}
\newcommand\bga{\overline\ga}

\newcommand\NNo{\set{0,1,2,\dots}}
\newcommand\dF{F} 
\newcommand\df{f} 

\newcommand\lir{\rho}

\newcommand\lirx{\hat\rho}
\newcommand\cirx{\hat c}
\newcommand\muu{\mu'}
\newcommand\On[1]{O\bigpar{n^{#1}}}
\newcommand\ci{c^{(1)}}
\newcommand\cii{c^{(2)}}
\newcommand\ciii{c^{(3)}}
\newcommand\civ{c^{(4)}}
\newcommand\cv{c^{(5)}}

\newcommand\gdii{\frac{\gd}{2}}
\newcommand\gdde{\gdd-\eps}
\newcommand\bzeta{\overline\zeta}
\newcommand\zetat{\zeta_{\ii t}}
\newcommand\gsst{\gss_{\ii t}}
\newcommand\cgs{\varsigma}
\newcommand\cgss{\cgs^2}
\newcommand\cxx[1]{_{(#1)}}
\newcommand\cx{\cxx{c}}

\begin{document}

\begin{abstract} 
For a complex number $\ga$, we consider the sum of the $\ga$th powers of
subtree sizes in Galton--Watson trees conditioned to be of size~$n$. Limiting distributions of
this functional $X_n(\alpha)$ have been determined for $\rea \neq 0$, revealing a
transition between a complex normal limiting distribution for $\rea < 0$ and
a non-normal limiting distribution for $\rea > 0$. In this
paper, we
complete the picture by proving a normal limiting
distribution, along with moment convergence, in the missing case $\rea =
0$. The same results are also established in the case of the so-called shape functional $X_n'(0)$,
which is the sum of the logarithms of all subtree sizes; these results were obtained earlier in special
cases.
Additionally, we prove
convergence of all moments in the case $\rea < 0$, where this result was previously
missing, and establish new results about the asymptotic mean for real $\ga < 1/2$. 

A novel feature for $\rea=0$ is that we find joint convergence for several
$\ga$ to independent limits, in contrast to the cases $\rea\neq0$, where the
limit is known to be a continuous function of $\ga$. Another difference from
the case $\rea\neq0$ is that there is a logarithmic factor in the asymptotic
variance when $\rea=0$; this holds  also for the shape functional.

The proofs are largely based on singularity analysis of generating functions. 
\end{abstract}

\maketitle

\ccreset
\section{Introduction and main results}\label{S:intro}

This paper is a sequel to \cite{SJ359}.
As there, we consider a conditioned \GWt{} $\tn$
and the random variables
\begin{align}
  X_n(\ga)&:=F_\ga(\tn):=\sum_{v\in\tn}|\tnv|^\ga,\label{Xn}
\end{align}
where $\tnv$ is the fringe subtree of $\tn$ rooted at a 
vertex $v\in\tn$, \ie, the subtree consisting of $v$ and all its descendants.
This is a special case of what is known as an additive functional:~a functional associated with a rooted tree $T$ that can be expressed in the form
\begin{align}
 F(T) = \sum_{v\in T} f(T_v)
\end{align}
for a certain toll function $f$. Thus, $F_\ga$ is the additive functional on rooted trees defined
by the 
toll function $f_\ga(T):=|T|^\ga$. 
For reasons discussed in \cite{SJ359}, we allow the parameter $\ga$ to be
any complex number.
(See further \refS{S:not} for the notation used here and below.)

In \cite{SJ359}, 
it is assumed that the \cGWt{} $\tn$ is 
defined by some offspring distribution $\xi$ with $\E\xi=1$ and
$0 < \gss := \Var\xi < \infty$.
The main results 
are limit theorems showing that then the random variables
$X_n(\ga)$ converge in distribution after suitable normalization.
The results differ for the two cases $\rea<0$ and $\rea>0$:
Typical results are the following (here somewhat simplified), where
\begin{align}
\tX_n(\ga):=X_n(\ga)-\E X_n(\ga).  
\end{align}
For further related results, and references to previous work,
see \cite{SJ359}.

\begin{theorem}[{\cite[Theorem 1.1]{SJ359}}]\label{TSJ359-1.1}
  If\/ $\rea<0$, then
  \begin{align}\label{t1.1}
    n\qqw\tX_n(\ga)\dto \hX(\ga),
  \end{align}
where $\hX(\ga)$ is a centered complex normal random variable with
distribution depending on the offspring distribution $\xi$.
\end{theorem}

\begin{theorem}[{\cite[Theorem 1.2]{SJ359}}]\label{TSJ359-1.2}
  If\/ $\rea>0$, then
  \begin{align}\label{t1.2}
    n^{-\ga-\frac12}\tX_n(\ga)\dto \gs\qw\tY(\ga),
  \end{align}
where $\tY(\ga)$ is a centered random variable with
a (non-normal) 
distribution that depends on $\ga$ but 
does not depend on the offspring distribution $\xi$.
\end{theorem}

Note the three differences between the two cases:
\begin{romenumerate}
  \item \label{sj359a} 
the normalization is by different powers of $n$, with the power constant for
$\rea<0$ but not for $\rea>0$;
\item\label{sj359b} 
the limit is normal for $\rea<0$ but not for $\rea>0$;
\item\label{sj359c} 
the limit distribution is universal for $\rea>0$ in the sense that it
depends on $\xi$ only by the scale factor $\gs\qw$, but for $\rea<0$, the 
distribution seems to depend on the offspring distribution $\xi$ in a more
complicated way. (In the latter case, the distribution is complex normal,
so it is determined by the covariance matrix of 
$\bigpar{\Re\hX(\ga),\Im\hX(\ga)}$; a complicated formula for covariances is
given in \cite[Remark 5.1]{SJ359}, but we do not know how to evaluate it for
concrete examples, not even 
when $\ga<0$ is real and thus $\hX(\ga)$ is a real random variable.)
\end{romenumerate}

The results above leave a gap: the case $\rea=0$, and the main 
purpose of the
present paper is to fill this gap, and to compare the results with the cases
above.
The case $\ga=0$ is trivial, since $X_n(\ga)=n$ is non-random.
However, in this case we instead study the derivative
\begin{align}\label{r'}
  X_n'(0) = \sum_{v\in \cT_n}\log|\cT_{n,v}|
=\log\prod_{v\in \cT_n}|\cT_{n,v}|,
\end{align}
which is known as the \emph{shape functional}.
This has earlier been studied in some special cases in \eg{}
\cite{Fill96,MM98,Pittel,FillK04,FillFK,Caracciolo20},
see \refS{S:shape}.

Another gap in \cite{SJ359} is that moment convergence was proved for
$\rea>0$ (\refT{TSJ359-1.2}) but not for
$\rea<0$ (\refT{TSJ359-1.1}). We fill that gap too.

For technical convenience, 
we assume throughout the paper the weak extra moment condition
\begin{align}\label{2+gd}
  \E\xi^{2+\gd}<\infty,
\end{align}
for some $\gd>0$; we also continue to assume $\E\xi=1$.
We let $\cT$ be an unconditioned \GWt{} with offspring distribution $\xi$,
and define, for complex $\ga$ with $\rea<\frac12$,
\begin{align}\label{mu}
\mu(\ga)&
:=
\E|\cT|^\ga
= \sumn n^\ga \P(|\cT|=n),
\\
\label{xd}
  \muu&:=\mu'(0)
=\E \log|\cT|
= \sumn  \P(|\cT|=n)\log n
.\end{align}
(The sum \eqref{mu} converges 
for $\rea<\frac12$, since $\P(|\cT|=n)=O(n^{-3/2})$; 
see \eqref{qno}.) 

Our main results are the following.
Note that $X_n'(0)$ is a real random variable, while $X_n(\ii t)$ 
and $X_n(\ga)$ for $\ga\notin\bbR$ are
non-real except in trivial cases.
As said above, special cases of \refT{TZ} have been proved by
\citet{Pittel}, 
\citet{FillK04}, 
and
\citet{Caracciolo20}. 

\begin{theorem}\label{TZ}
Assume \eqref{2+gd} with $\gd>0$.
Then,
\begin{align}\label{tz}
\frac{  X_n'(0)-\mu'n}{\sqrt{n\log n}}\
\dto N\bigpar{0,4(1-\log2)\gs\qww}
\end{align}
together with convergence of all moments.
\end{theorem}

\begin{theorem}\label{Tit}
Assume \eqref{2+gd} with $\gd>0$.
Then, for any real $t\neq0$,
\begin{align}\label{tit}
\frac{X_n(\ii t)-\mu(\ii t)n}{\sqrt{n\log n}}\
\dto \zetat
\end{align}
together with convergence of all moments,
where $\zetat$ is a symmetric complex normal variable with
variance
\begin{align}\label{tit2}
  \E|\zetat|^2=
\frac{1}{\sqrt\pi}\Re\frac{\gG(\ii t-\half)}{\gG(\ii t)}\gs\qww
>0
.\end{align}
\end{theorem}

\begin{theorem}\label{TNEG}
Assume \eqref{2+gd} with $\gd>0$.
Then, for any complex $\ga$ with $\rea<0$,
\begin{align}\label{tneg}
\frac{X_n(\ga)-\mu(\ga)n}{\sqrt{n}}
\dto \hX(\ga)
\end{align}
together with convergence of all moments,
where $\hX(\ga)$ is
a centered complex normal random variable with
positive variance and
distribution depending on the offspring distribution $\xi$.
Hence, \eqref{t1.1} holds with convergence of all moments.
\end{theorem}

\begin{remark}\label{Rmom}
  By ``convergence of all moments'', we mean in the case of complex variables,
  $Z_n$ say, convergence of all mixed moments of $Z_n$ and $\overline {Z_n}$,
which is equivalent to convergence of all mixed moments of $\Re Z_n$ and
$\Im Z_n$.
Since we have convergence in distribution, this is by a standard
argument using uniform integrability also equivalent to convergence of all
absolute moments.

Note that, conversely, by the method of moments 
applied to $(\Re Z_n,\Im Z_n)$,
this implies convergence in distribution of
$Z_n$, provided, as is the case here, the limit distribution is determined
by its moments. Thus, our proof of moment convergence provides a new proof
of \refT{TSJ359-1.1}, very different from the proof in \cite{SJ359}.
\end{remark}

\begin{remark}\label{Rcenters}
Since the statements include convergence of the first moments (to 0), 
we may in \refTs{TZ}--\ref{TNEG} replace $\mu'n$, $\mu(\ii t)n$, and 
$\mu(\ga)n$ by 
the expectations $\E X_n'(0)$, $\E X_n(\ii t)$, and $\E X_n(\ga)$, respectively;
in particular, this gives the last sentence in \refT{TNEG}.
More precise estimates of the expectations are given in \eqref{lz1b},
\eqref{lz1b++}, 
\eqref{ixq}, and \eqref{ixq-}.
\end{remark}

\refTs{TZ} and \ref{Tit} 
combine some of the features found for $\rea<0$
and $\rea>0$ in \ref{sj359a}--\ref{sj359c} above.
First, the variances in \refTs{TZ} and \ref{Tit}
are of order $n\log n$. This might be a surprise since  it is not what a naive
extrapolation from either $\rea<0$ in \refT{TSJ359-1.1} or $\rea>0$ 
in \refT{TSJ359-1.2} would yield, where the variances are of order 
$n$ ($\rea<0$) and  $n^{1+2\rea}$ ($\rea>0$); however, it is not surprising
that a logarithmic  factor appears when the two different expressions meet.
Secondly, the limits are normal, as heuristically would be expected by 
``continuity'' from the left, see \ref{sj359b}.
Thirdly, the limits are universal and depend only on $\gs$ as a scale
factor,
as heuristically would be expected by 
``continuity'' from the right, see \ref{sj359c}.

The proofs in \cite{SJ359} use two different methods, which are combined to
yield the full results: (1) methods using complex analysis and the fact
that $X_n(\ga)$ is an analytic function of $\ga$, and (2) analysis of
moments for a fixed $\ga$ using singularity analysis of generating functions
based on results
of \citet{FillFK}, also presented in \cite[Section VI.10]{FS}.
In the present paper, we will use only the second method.
We follow the proofs in \cite{SJ359} with some variations
(see also \cite{FillK03} and \cite{FillK04}).
However, some new leading terms will appear in the singular
expansions of the generating functions, which
will dominate the terms that are leading in \cite{SJ359};
this explains both the
logarithmic factors in the variance (and  in higher moments)
in \refTs{TZ} and \ref{Tit}, and the fact that these theorems 
yield normal limits
while 
\refT{TSJ359-1.2} 
does not.

After some preliminaries in \refS{S:not},
we first study the shape functional and prove \refT{TZ} in \refS{S:shape};
we then study the case of imaginary exponents and prove \refT{Tit} in
\refS{S:I};
after that, we consider the case $\rea<0$ and prove \refT{TNEG} in \refS{NEG}.
These three sections use the same  method 
(from \cite{FillK04} and \cite{SJ359}), 
and are thus quite similar,
but some details differ. 
The differences arise
partly because $X_n'(0)$ is real, while $X_n(\ii t)$
and (in general) $X_n(\ga)$
are not; we will also see that the logarithmic factors in the first two
cases appear in the moments in somewhat different ways,
and that there is a cancellation of some leading terms in our induction
for the first and third case, but not for $X_n(\ii t)$.
For this reason, we give complete arguments for all three cases,
and we encourage the reader to compare them and see both similarities and
differences.

In \refS{S:comp} we show how the centering functions \eqref{mu} and \eqref{xd}
can be compared across variation in the offspring distribution when (real) $\ga$
satisfies $\ga < \frac12$.

\begin{remark}\label{Rjoint}
  The results in \cite{SJ359} show also joint convergence for different
  $\ga$ in \refTs{TSJ359-1.1} and \ref{TSJ359-1.2}, 
with limits $\hX(\ga)$ and $\tY(\ga)$ that are analytic, and in
  particular continuous, random functions of the parameter $\ga$ in the
  half-planes $\rea<0$ and $\rea>0$, respectively.
This does \emph{not} extend to the imaginary axis $\rea=0$; we will see 
in \refT{Tjoint} that
$X_n(\ga)$ for different imaginary $\ga$ are asymptotically independent (for
$\Im\ga>0$), 
and thus it is not possible to have joint convergence to a continuous random
function. 
\end{remark}

\begin{remark}\label{R=}
  Let $\ga=s+\ii t$, where $t$ is real and fixed, and let $s\downto0$.
(Thus  $s>0$ is real.) 
It is shown in \cite[Appendix D]{SJ359} that if $t\neq0$, then the limit
$\tY(s+\ii t)$ diverges (in probability, say) as $s\downto0$, and that
$s\qq\tY(s+\ii t)\dto\zeta$, 
where $\zeta$ is a symmetric complex normal variable with
\begin{align}\label{ri2}
    \E|\zeta|^2
=\frac{1}{2\sqrt\pi} \Re\frac{\gG(\ii t-\frac12)}{\gG(\ii t)}>0.
\end{align}
(However,  unfortunately there is a typo in \cite[(D.2)]{SJ359},
see \refApp{Atypo}.)
Similarly, it is shown in \cite[Appendix C]{SJ359} 
that 
$s\qqw\tY(s)\dto N(0,2(1-\log2))$ as $s\downto 0$;
in particular $s\qw\tY(s)$ 
diverges.

These results may be compared to \refTs{TZ}--\ref{Tit};
note that 
the limits are the same,
except that the variances in both cases differ by a factor $1/2$
(which of course depends on the chosen normalizations).
Both sets of results
can be regarded as
 iterated 
limits of $\tX_n(s+\ii t)$, taking \ntoo{} and
$s\downto0$ in different orders. The divergence of $\tY(s+\ii t)$ as $s\downto0$
(for fixed $t\neq 0$) thus seems to be related to the fact that the
asymptotic variance in \refT{TZ} is
 of greater order
than $n$, and similarly
the divergence as $s\downto0$
of $s\qw\tY(s)$ 
(which loosely might be regarded as an approximation of $n\qqw \tX_n'(0)$)
seems related to \refT{Tit}.
However, we do not see why the factors $s^{\pm\frac12}$ in these limits
should correspond to the
factor  $(\log n)^{1/2}$ in \refTs{TZ} and \ref{Tit}
[or more precisely 
to the factor $(2 \log n)^{1/2}$,
to get exactly the same limit distributions].
\end{remark}

We end with some problems suggested by the results and comments above.

\begin{problem}\label{P=}
Is 
there a simple explanation of the equality discussed in \refR{R=}
of 
iterated
limits in different orders, but with different normalizations?
Is this an instance of some general phenomenon?
What happens if $s\downto0$ and \ntoo{} simultaneously, i.e., 
for
$\tX_n(s_n+\ii t)$ where $s_n\downto0$ at  some appropriate rate?
\end{problem}

The asymptotic independence of $X(\ii t)$ for $t>0$ mentioned in
\refR{Rjoint} suggests informally that the stochastic process 
$(\tX_n(\ii t): t\ge0)$ asymptotically looks something like white noise. 
This might be investigated further, for example as follows.
\begin{problem}\label{Pwhite}
Consider the integrated process $\int_0^t \tX_n(\ii u)\dd u$.
What is the order of its variance?
Does this process after normalization
converge to a process with paths that are continuous in~$t$?
\end{problem}

The moment assumption \eqref{2+gd} is used repeatedly to control error
terms, but it seems convenient rather than necessary.

\begin{problem}
  We conjecture that  \refTs{TZ}--\ref{TNEG} hold also without the
  assumption \eqref{2+gd}. Prove (or disprove) this!
\end{problem}

\section{Notation and preliminaries}\label{S:not}
\subsection{General notation}\label{Sgennot}
As said above, $\cT$ is a \GWt{} 
defined by an offspring
distribution $\xi$ with mean $\E\xi=1$ and finite non-zero variance
$\gss:=\Var\xi<\infty$, and  we assume \eqref{2+gd} for some $\gd>0$. 
Furthermore, the \cGWt{} $\tn$ is defined as $\cT$ conditioned on $|\cT|=n$.
We assume for simplicity that $\xi$ has span 1; the general case follows by
standard (and minor) modifications.

$\gG(z)$ denotes the Gamma function,  $\psi(z):=\gG'(z)/\gG(z)$ is its
logarithmic derivative, and $\gamma=-\psi(1)$ is Euler's constant.

A random variable $\zeta$ has a \emph{complex normal distribution}
if it
takes values in $\bbC$ and $(\Re\zeta,\Im\zeta)$ is a 2-dimensional
normal distribution (with arbitrary covariance matrix).
In particular,
$\zeta$ is \emph{symmetric complex normal}
if further
$\E\zeta=0$ and
$(\Re\zeta,\Im\zeta)$ has covariance matrix
$\smatrixx{\cgss/2&0\\0&\cgss/2}$ for some $\cgss=\E|\zeta|^2$,
which is called the \emph{variance};
equivalently, $\E\zeta=0$, $\E\zeta^2=0$, and $\E|\zeta|^2=\cgss$.
(See \eg{} \cite[Proposition 1.31]{SJIII}.)
A symmetric complex normal distribution with variance $\cgss$
is determined by the mixed moments
of $\zeta$ and $\bzeta$, which are given by
(see \cite[p.~14]{SJIII})
\begin{align}\label{cnormal}
  \E\bigsqpar{\zeta\strut^\ell\,\bzeta^r}=
  \begin{cases}
   \cgs^{2\ell}\ell!, & \ell=r,
\\
0,&\ell\neq r.
  \end{cases}
\end{align}

Unspecified limits are as \ntoo.
We let
$\dto$ denote convergence in distribution.

For real $x$ and $y$, we denote $\min(x,y)$ by $x\land y$.

The semifactorial $\ell!!$ is defined for odd integers $\ell$ (the only case
that we use) by
\begin{align}\label{semi}
  \ell!! := 1\cdot 3 \cdot \, \cdots \, \cdot \ell
=2^{(\ell+1)/2} \gG\lrpar{\tfrac{\ell}{2}+1}/\sqrt\pi.
\end{align}
Note that $(-1)!!=1!!=1$.

$\eps$ denotes an arbitrarily small fixed number with $\eps>0$.
(We will tacitly assume that $\eps$ is sufficiently small when necessary.)

$C$ and $c$ denote unimportant positive constants, possibly different each
time; these may depend on the parameter $\ga$ (or $\ga_1,\ga_2$ below).
We sometimes use $c$ with subscripts; these keep the same value
within the same section.

\subsection{\GDD{s} and singularity analysis}\label{SSGDD}
A \GDD{} is a complex domain of the type
\begin{align}\label{GDD}
\set{z:|z|<R,\, z\neq1,\,|\arg(z-1)|>\gth}  
\end{align}
where $R>1$ and $0<\gth<\pi/2$, see \cite[Section VI.3]{FS}.
A function is \emph{\gda} if it is analytic in some \GDD{}
(or can be analytically continued to such a domain). 

Our proofs are based on singularity analysis of various generating functions
(see \cite[Chapter VI]{FS}),
using estimates as
$z\to1$ in a suitable \GDD;
the domain may be different each time.
All estimates below of analytic functions tacitly are 
valid in some \GDD{s} 
(possibly different ones for different functions), 
even when that is not said explicitly.

\subsection{Polylogarithms}
$\Li_\ga(z)$ and $\Li_{\ga,r}(z)$ denote polylogarithms and generalized
polylogarithms, respectively; 
they are defined for $\ga\in\bbC$ and $r=0,1,\dots$
by the power series 
\begin{align}\label{lix}
  \Li_\ga(z)&:=\sumn n^{-\ga}z^n,
\\
\label{lixr}
  \Li_{\ga,r}(z)&:=\sumn (\log n)^r\frac{z^n}{n^\ga}
\end{align}
for $|z|<1$, 
and then extended analytically to $\bbC\setminus\ooo$
(in particular they are \gda); 
see   \eg{} \cite[Section VI.8]{FS}.
Note that $\Li_{\ga,0}(z)=\Li_\ga(z)$.
We will also use the notation
\begin{align}\label{Lz}
  L(z):=-\log(1-z)
=\sumn\frac{z^n}{n}
=\Li_{1}(z)
.\end{align}

We will use
singular expansions of polylogarithms and generalized polylogarithms 
into powers of $1-z$, possibly including powers of $L(z)$.
Infinite singular expansions of polylogarithms and generalized polylogarithms 
are given by \citet[Theorem 1]{Flajolet1999}
(also \cite[Theorem VI.7]{FS}); we will
mainly use only the following simple versions keeping only the main terms.

For any real $a$, let 
$\cP_a$ be the set of all polynomials in $z$ of degree $<a$.
In particular, if $a\le0$, then $\cP_a=\set0$. If $0\le a\le1$, then every
polynomial in $\cP_a$ is constant. These simple cases are the ones
of most interest to us.

We then have, for each $\ga\notin\set{1,2,\dots}$,
\begin{align}\label{lia}
\Li_{\ga}(z) = \gG(1-\ga)(1-z)^{\ga-1} +P(z)+ O\bigpar{|1-z|^{\rga}},
\end{align}
for some $ P(z)\in\cP_{\rga}$.

Moreover, in our proofs we will often go back and forth between expansions
in powers of $1-z$ (including powers of $L(z)$)
and expansions in (generalized) polylogarithms, using the following simple
consequence of the singular expansions of generalized polylogarithms, proved
in \cite{FillK04}. (Here slightly simplified.)

\begin{lemma}[{\cite[Lemmas 2.5--2.6]{FillK04}}]\label{Liar}
Suppose that $\rea<1$. 
Then,
for each $r\ge0$,
in any fixed \GDD{} and for any  $\eps>0$, 
\begin{align}\label{liar1}
  \Li_{\ga,r}(z) 
&= \sum_{j=0}^r \lir_{r,j}(\ga) (1-z)^{\ga-1} L(z)^j + c_r(\ga)
+\Oz{\Re \ga-\eps},
\end{align}
for some coefficients $\lir_{r,j}(\ga)$
and $c_r(\ga)$, 
with leading coefficient
\begin{align}\label{liar2}
  \lir_{r,r}(\ga)
=\gG(1-\ga).
\end{align}
Conversely, 
\begin{align}\label{liar0}
(1-z)^{\ga-1} L(z)^r
&= \sum_{j=0}^r \lirx_{r,j}(\ga)  \Li_{\ga,j}(z) +\cirx_r(\ga)
+\Oz{\Re \ga-\eps},
\end{align}
for some coefficients $\lirx_{r,j}(\ga)$ and $\cirx_r(\ga)$,
with
\begin{align}\label{liar20}
  \lirx_{r,r}(\ga)
=\lir_{r,r}(\ga)\qw=\gG(1-\ga)\qw.
\end{align} 
\end{lemma}

\begin{remark}
  The lemmas in \cite{FillK04} are stated for real $\ga$, but the proofs
  hold also for complex $\ga$. Moreover, the results extend to $\ga$ with
$\rea\ge1$, assuming 
$\ga\notin\set{1,2,\dots}$,
provided the error terms $\Oz{\Re \ga-\eps}$  are replaced by $O(|1-z|)$ when
$\rea>1$.
\end{remark}

\subsection{Hadamard products}
Recall  that the \emph{Hadamard product} $A(z)\odot B(z)$
of two power series
$A(z)=\sumno a_n z^n$ and $B(z)=\sumno b_n z^n$
is defined by 
\begin{align}\label{hadamard}
  A(z)\odot B(z) := \sumno a_n b_n z^n.
\end{align}
As a simple example, for any complex $\ga$ and $\gb$,
\begin{align}\label{lili}
  \Li_\ga(z)\odot\Li_\gb(z)=\Li_{\ga+\gb}(z),
\end{align}
and, more generally, by \eqref{lixr},
\begin{align}\label{lilix}
  \Li_{\ga,r}(z)\odot\Li_{\gb,s}(z)=\Li_{\ga+\gb,r+s}(z).
\end{align}
We note also,  for any constant $c$ and 
power series $A(z)=\sumno a_nz^n$,
the trivial result
\begin{align}\label{lilic}
  c\odot A(z) = ca_0.
\end{align}

For our error terms, we will use the following lemma;
it is part of \cite[Lemma 12.2]{SJ359} and
taken from \cite[Propositions 9 and 10(i)]{FillFK}
and \cite[Theorem VI.11 p.~423]{FS}.
(Further related results are given in \cite{FillFK}, 
\cite[Section VI.10.2]{FS} and
\cite{SJ359}.) 

\begin{lemma}[\cite{FillFK,FS}]\label{LFFK}
  If $g(z)$ and $h(z)$ are \gda{}, then $g(z)\odot h(z)$ is \gda.
Moreover, suppose that
$g(z)=O(|1-z|^a)$ and $h(z)=O(|1-z|^b)$, where
$a$ and $b$ are real with $a+b+1\notin\NNo$; then, as $z\to1$ in a suitable
\GDD, 
\begin{align}
  \label{lffkO}
g(z)\odot h(z)=P(z)+\Oz{a+b+1},
\end{align}
for some $ P(z)\in\cP_{a+b+1}$.
\end{lemma}

\subsection{Generating functions for \GWt{s}}\label{SSGW}

Let $p_k:=\P(\xi=k)$ denote the values of the probability mass function
for the offspring distribution $\xi$,
and let $\Phi$ be its \pgf:
\begin{align}\label{Phi}
  \Phi(z):=\E z^\xi=\sumko p_k z^k.
\end{align}
Similarly, let $q_n:=\P(|\cT|=n)$, and let~$y$ denote the corresponding
\pgf:
\begin{align}\label{yz}
  y(z):=\E z^{|\cT|} = \sumn \P\bigpar{|\cT|=n}z^n
=\sumn q_n z^n.
\end{align}
As is well known, then
\begin{align}\label{Phiy}
  y(z) &
=z\Phi\yz
.\end{align}

Under our assumptions $\E\xi=1$ and $0<\Var\xi<\infty$,
the generating function
$y(z)$ extends analytically to a \GDD{} and is thus \gda;
see 
\cite[Lemma A.2]{SJ167}
and \cite[\S12.1]{SJ359}
(and under stronger  assumptions \cite[Theorem VI.6, p.~404]{FS}).
Furthermore, 
see again \cite[Lemma A.2]{SJ167},
there exists a \GDD{} where $|y(z)|<1$, and thus
$\Phi(y(z))$ is \gda, as well as 
$\Phi^{(m)}\bigpar{y(z)}$ for every $m\ge1$.

We note some useful consequence of our extra moment assumption \eqref{2+gd};
we may without loss of generality assume $\gd\le1$.
(Compare \cite[(12.5), (12.30), and (12.31)]{SJ359} 
without the assumption \eqref{2+gd} but with weaker error terms,
and
\cite[Theorem VI.6]{FS} with stronger results under stronger assumptions.)

\begin{lemma}\label{Lgdy}
  If \eqref{2+gd} holds with $0<\gd\le1$, then, for $z$ in some \GDD,
  \begin{align}\label{lgdy}
  y(z)&=1-\sqrt2 \gs\qw(1-z)\qq+\Oz{\frac12+\gdd},
\\\label{lgdy'}
  y'(z)&=2\qqw \gs\qw(1-z)\qqw+\Oz{-\frac12+\gdd}.
\\\label{lgdyy}
\frac{zy'(z)}{y(z)}&=2\qqw \gs\qw(1-z)\qqw+\Oz{-\frac12+\gdd},
  \end{align}
In particular, all three functions are \gda.
\end{lemma}

\begin{proof}
That $y(z)$ is \gda{} was noted above, and
the estimate \eqref{lgdy} was shown in \cite[Lemma 12.15]{SJ359}.
A differentiation then yields \eqref{lgdy'} in a smaller \GDD,
using Cauchy's estimates for a disc with radius
$c|1-z|$ centered at $z$
(see \cite[Theorem VI.8 p.~419]{FS}).

Note that $zy'(z)/y(z)$ is analytic in any domain where $y$ is 
defined and analytic with $|y(z)|<1$, 
since then \eqref{Phiy} holds in the domain and implies that $y(z)\neq0$ for
$z\neq0$, and also that $z/y(z)$ is analytic at $z=0$.
Hence, also $zy'(z)/y(z)$ is \gda.
Finally,  \eqref{lgdyy} follows from \eqref{lgdy} and \eqref{lgdy'}.
\end{proof}

By \eqref{lia},  and using $\gG(-1/2)=-2\sqrt\pi$, we can rewrite \eqref{lgdy} as
\begin{align}\label{yli}
    y(z)&=-\frac{\sqrt2}{\gG(-\frac12) \gs}\Li_{3/2}(z)+c+\Oz{\frac12+\gdd}
\notag\\&
=\frac1{\sqrt{2\pi}\gs}\Li_{3/2}(z)+c+\Oz{\frac12+\gdd},
\end{align}
where 
(although we do not need it)
$c=1-\zeta(3/2)/\sqrt{2\pi\gss}$.
Furthermore, by \eqref{lix} and
singularity analysis
\cite[Theorem VI.3, p.~390]{FS},
\eqref{yli} implies 
\begin{align}
  \label{qn}
q_n=\P(|\cT|=n)= \frac{1}{\sqrt{2\pi}\gs}n^{-3/2}+\On{-\frac32-\gdd}
= \frac{1+\On{-\gd / 2}}{\sqrt{2\pi}\gs}n^{-3/2}
.\end{align}

\begin{remark}
It is well known that  the  asymptotic formula 
\begin{align}
\label{qno}
q_n=\P(\cT|=n)= \frac{1+o(1)}{\sqrt{2\pi}\gs}n^{-3/2}
\qquad\text{as \ntoo},
\end{align}
with a weaker error bound than \eqref{qn},
holds assuming only
$\Var\xi<\infty$ (and $\E\xi=1$);
see \eg{}
\cite{Otter} (assuming an exponential moment),
\cite[Lemma 2.1.4]{Kolchin} or 
\cite[Theorem 18.11]{SJ264}
(with $\tau=\Phi(\tau)=1$)
and the further references given there.
\end{remark}

\begin{lemma}\label{LZ0}
Assume \eqref{2+gd} with $0<\gd\le1$. 
Then, for $z$ in some \GDD,
\begin{align} \label{lz00}
  \Phi\bigpar{y(z)}&=1 + \Oz{\frac12},
\\\label{lz0a}
  \Phi'\bigpar{y(z)}&=1 + \Oz{\frac12},
\\\label{lz0b}
  \Phi''\bigpar{y(z)}&= \gss + \Oz{\gdd},
\intertext{and, for each fixed $m\ge3$,}
\label{lz0c}
  \Phi^{(m)}\bigpar{y(z)}&= \Oz{\gdd+1-\frac{m}2} 
.\end{align}
\end{lemma}

\begin{proof}
The assumption \eqref{2+gd} implies the estimate, see \eg{} 
\cite[Lemma 12.14]{SJ359},
  \begin{align}\label{lgd}
    \Phi(z)=z+\tfrac12\gss (1-z)^2 + O\bigpar{|1-z|^{2+\gd}}, 
\qquad |z|\le1.
  \end{align}
By differentiation of \eqref{lgd}, for the remainder term
using Cauchy's estimates 
for a disc with radius $(1-|z|)/2$ centered at $z$, 
we obtain for all $z$ with $|z|<1$,
and each fixed $m\ge3$,
\begin{align}\label{za}
  \Phi'(z)&=1 - \gss(1-z) + O\bigpar{|1-z|^{2+\gd}/(1-|z|)}, 
\\\label{zb}
  \Phi''(z)&= \gss + O\bigpar{|1-z|^{2+\gd}/(1-|z|)^2},
\\\label{zc}
  \Phi^{(m)}(z)&= O\bigpar{|1-z|^{2+\gd}/(1-|z|)^m}.
\end{align}

For $z$ in a suitable \GDD{}  we have \eqref{lgdy}, 
and as a consequence, if $|1-z|$ is small enough,
  \begin{align}\label{zy}
c|1-z|\qq \le 1-|y(z)|   \le |1-y(z)| \le C|1-z|\qq.
  \end{align}
The result follows by \eqref{lgd}--\eqref{zy}.
\end{proof}

\begin{remark}
In fact, \eqref{lz0a} holds without the extra assumption \eqref{2+gd},
assuming only $\E\xi^2<\infty$, 
because then $\Phi$ is twice continuously
differentiable in the closed unit disc with $\Phi'(1)=1$, and
$y(z)=1-\sqrt2\gs\qw(1-z)\qq+o(|1-z|\qq)$
as is shown in \cite[Lemma A.2]{SJ167}, see also \cite[(12.5)]{SJ359}.
\end{remark}

\section{The shape functional}\label{S:shape}

\ccreset

We consider here the shape functional $X_n'(0)$.
Asymptotics for the mean and variance were found by \citet{Fill96} in the
case of uniform binary trees [the case $\xi\sim\Bi(2,\frac12)$];
this was generalized by \citet{MM98} to simply generated trees under a
condition equivalent to our \cGWt{s} with $\xi$ having a finite exponential 
moment $\E e^{r\xi}<\infty$ for some $r>0$.
\citet{Pittel} showed asymptotic normality in the case of uniform labelled
trees [the case $\xi\sim\Po(1)$] by estimating cumulants.
\citet{FillK04} considered uniform binary trees 
[$\xi\sim\Bi(2,\frac12)$] and showed asymptotic normality 
by estimating moments by singularity analysis, see also
\citet{FillFK}.
Asymptotic normality has recently been shown, by similar methods, also
for uniformly random ordered trees 
[the case $\xi\sim\Ge(\half)$] by \citet{Caracciolo20}, 
who further [personal communication]
have extended the results to arbitrary offspring distributions $\xi$ (with
$\E\xi=1$ as here), 
at least provided  that $\xi$ has a finite exponential moment 
$\E e^{r\xi}<\infty$ for some $r>0$.

We will here extend these results to any offspring distribution $\xi$
satisfying the standard condition $\E\xi=1$ and 
the weak moment condition \eqref{2+gd} for some $\gd>0$. 
We assume without loss of generality that $0<\gd\le1$.
We will use singularity analysis to estimate moments, in the same
way as \cite{FillK04,FillFK,Caracciolo20}.

In this section, we define
(corresponding to \cite[(12.46)]{SJ359})
\begin{align}\label{xb}
  b_n:=\log n -\muu, \qquad n\ge1,
\end{align}
where $\muu
=\E \log|\cT|
= \sumn q_n\log n
$ as in \eqref{xd},
and 
we let $F$ be the additive functional defined by the toll function
$f(T):=b_{|T|}$. Thus, by \eqref{r'},
\begin{align}\label{xc}
  F(\ctn)=X_n'(0)-\muu n
.\end{align}

The generating function of $b_n$ is, 
by \eqref{lix}--\eqref{lixr} 
and noting  $\Li_0(z)=z/(1-z)$,
\begin{align}\label{xe}
  B(z) = \sumn(\log n-\muu)z^n 
= \Li_{0,1}(z)-\muu\Li_0(z)
.\end{align}
Hence, by \refL{Liar}
(or \cite[Figure VI.11, p.~410]{FS}  with more  terms),
\begin{align}\label{Bz1}
B(z) &= (1-z)\qw L(z) -c(1-z)\qw+\Oz{-\eps}
\\\label{Bz2}
&=\Oz{-1-\eps}.
\end{align}

We define the generating functions, for $\ell\ge1$,
\begin{align}\label{Mell}
  M_\ell(z):=\E \bigsqpar{F(\cT)^\ell z^{|\cT|}}
=\sumn q_n \E[F(\tn)^\ell] z^n
.\end{align}
These generating functions  can be calculated 
recursively
by the following formula (valid for any sequence $b_n$) 
from \cite{SJ359}.

\begin{lemma}[{\cite[Lemma 12.4]{SJ359}}]
  \label{LH}
For every $\ell\ge1$,
\begin{align}\label{lh}
  M_\ell(z)
=
\frac{z y'(z)}{y(z)}
\summl \frac{1}{m!}\sumxx \binom{\ell}{\ell_0,\dots,\ell_m}
B(z)^{\odot\ell_0} \odot 
\bigsqpar{zM_{\ell_1}(z)\dotsm M_{\ell_m}(z)\Phi\xxm\yz},
\end{align}
where $\sumxx$ is the sum over all $(m+1)$-tuples $(\ell_0,\dots,\ell_m)$ of
non-negative integers summing to $\ell$ such that 
$1\le\ell_1,\dots,\ell_m<\ell$.
\end{lemma}

Note that $B(z)$ is \gda{} by \eqref{xe};
furthermore, $zy'(z)/y(z)$ and $\Phi^{(m)}\bigpar{y(z)}$
are also \gda, see \refSS{SSGW}. 
Hence, 
\eqref{lh} and induction using \refL{LFFK}
show that every $M_\ell(z)$ is \gda.
  
It will be convenient to denote the sum in \eqref{lh} by $R_\ell(z)$.
Thus,
\begin{align}\label{lh6''}
  M_\ell(z)=\frac{zy'(z)}{y(z)}R_\ell(z).
\end{align}

\subsection{The mean}

We begin with the mean $\E X'_n(0)$ and the corresponding generating
function $M_1(z)$.
The following result includes 
earlier results for special cases in 
\cite{Fill96,MM98,Pittel,FillK04,FillFK,Caracciolo20}, 
but our error term is weaker 
[since we have the weaker moment assumption \eqref{2+gd}].
Recall that $\psi(z):=\gG'(z)/\gG(z)$, and note that
\begin{align}\label{psi-1/2}
 \psi(-\tfrac12)=\psi(\tfrac12)+2=-\gam-2\log 2+2, 
\end{align}
see \cite[5.5.2 and 5.4.13]{NIST}.

\begin{lemma}\label{LZ1}
Assume \eqref{2+gd} with $0<\gd\le1$.
Then, for any $\eps>0$,
\begin{align}\label{lz1a}
  M_1(z) = 
-\gs\qww L(z) 
+\frac{{\muu-\psi(-\frac12)}}{\gss} 
+ \Oz{\gdd-\eps}
\end{align}
and
\begin{align}\label{lz1b}
  \E X'_n(0)= 
\muu n - \frac{\sqrt{2\pi}}{\gs} n\qq + \On{\frac12-\gdd+\eps}.  
\end{align}
\end{lemma}
\begin{proof}
For $\ell=1$, the sums in \eqref{lh} reduce to a single term with $m=0$ and
$\ell_0=1$, and thus, 
as in \cite[(12.29)]{SJ359},
using \eqref{Phiy},
\begin{align}\label{xsx1}
  M_1(z)
=\frac{zy'(z)}{y(z)}\cdot\bigpar{ B(z)\odot z\Phi\yz}
=\frac{zy'(z)}{y(z)}\cdot\bigpar{ B(z)\odot y(z)}
.\end{align}

 By \eqref{lilix}, \eqref{lilic}, \eqref{xe} and  \eqref{yli} we obtain, 
using \refL{LFFK} and \eqref{Bz2} for the error term,
\begin{align}\label{xg}
  B(z)\odot y(z)
= 
\frac1{\sqrt{2\pi}\gs}\Li_{3/2,1}(z)
-\frac{\muu}{\sqrt{2\pi}\gs}\Li_{3/2}(z)
+\cc+\Oz{\frac12+\gdd-\eps}.
\end{align}
Further, by our choice \eqref{xd} of  $\muu$,
\begin{align}\label{xh}
  (B\odot y)(1) = \sumn b_n q_n 
=\sumn q_n(\log n -\muu) = \muu-\muu=0.
\end{align}
By \eqref{lia},
we have
\begin{align}\label{xi}
  \Li_{3/2}(z) = \gG(-\tfrac12)(1-z)\qq  + \cc + \Oz{}.
\end{align}
Moreover,
by  \cite[Theorem VI.7, p.~408]{FS} (or \cite[Theorem 1]{Flajolet1999}),
\begin{align}\label{xj}
  \Li_{3/2,1}(z) = \gG(-\tfrac12)(1-z)\qq L(z) + \gG'(-\tfrac12)(1-z)\qq + 
\cc + \Oz{}
.\end{align}
Hence, \eqref{xg} and \eqref{xi}--\eqref{xj} yield, 
using \eqref{xh} to see that the constant terms
cancel,
\begin{align}\label{xk}
  B(z)\odot y(z) &= 
\frac{\gG(-\frac12)}{\sqrt{2\pi}\gs}(1-z)\qq L(z) + 
\frac{\gG'(-\frac12)-\muu\gG(-\frac12)}{\sqrt{2\pi}\gs}(1-z)\qq 
+ \Oz{\frac12+\gdd-\eps}
\notag\\&
=-\frac{\sqrt{2}}{\gs}(1-z)\qq L(z) 
+\frac{\sqrt{2}\bigpar{\muu-\psi(-\frac12)}}{\gs} (1-z)\qq 
+ \Oz{\frac12+\gdd-\eps}
.\end{align}
 Finally, \eqref{xsx1}, \eqref{lgdyy}, and \eqref{xk} yield \eqref{lz1a}.

Since $L(z)=\sumn z^n/n$, \eqref{lz1a} yields by standard singularity analysis,
recalling the definition \eqref{Mell},
\begin{align}\label{xn}
q_n \E F(\ctn) =-\gs\qww n\qw +\On{-1-\gdd+\eps}.
\end{align}
Hence, using also \eqref{qn},
\begin{align}\label{xp}
  \E F(\ctn)=-\frac{\sqrt{2\pi}}{\gs}n\qq+\On{\frac12-\gdd+\eps}
\end{align}
and \eqref{lz1b} follows by \eqref{xc}.
\end{proof}

\begin{remark}
Under  stronger moment conditions on the offspring distribution $\xi$,
we may in the same way obtain an expansion of the mean $\E X_n'(0)$ with
further terms.
For example, if $\E \xi^{3+\gd}<\infty$, 
then the same argument yields
\begin{align}\label{lz1b++}
  \E X'_n(0)= 
\muu n - \frac{\sqrt{2\pi}}{\gs} n\qq 
+ \frac{\E[\xi(\xi-1)(\xi-2)]}{3\gs^4} \log n 
+O(1).
\end{align}
In the special case of binary trees,
this was given in \cite[(4.2)]{FillK04}.
Note that the coefficient of $\log n$ in \eqref{lz1b++} vanishes for
binary trees, but not in general.
\end{remark}

\subsection{The second moment}

\begin{lemma}\label{LZ2}
Assume \eqref{2+gd} with $0<\gd\le1$.
Then, for any  $\eps>0$,
\begin{align}\label{lz2}
M_2(z) &= 
2^{3/2}(1-\log 2)\gs^{-3} (1-z)\qqw L(z)
+\cc(1-z)\qqw + \Oz{-\frac12+\gdd-\eps}
.\end{align}
\end{lemma}

\begin{proof}
We use \refL{LH} with the notation $R_\ell(z)$ as in \eqref{lh6''}.
For $\ell=2$, \refL{LH}  shows, using \eqref{Phiy}, that
\begin{align}\label{y11}
  R_2(z)&=
B(z)^{\odot2}\odot y(z)
+2B(z)\odot\bigsqpar{zM_1(z)\Phi'\xpar{y(z)}}
+zM_1(z)^2\Phi''\bigpar{y(z)}.
\end{align}
We consider the three terms separately.

First, by \eqref{Bz2}, \eqref{lgdy} and \refL{LFFK} (twice), we have
\begin{align}\label{y12}
  B(z)^{\odot2}\odot y(z)
= B(z)^{\odot2}\odot \bigpar{y(z)-1}
=\cc + \Oz{\frac12-2\eps}.
\end{align}

For the remaining two terms, we have to be more careful, since it will turn
out that their main terms cancel.

For the second term, 
we note first that  \eqref{lz1a} implies $M_1(z)=\Oz{-\eps}$,
and thus \eqref{lz0a} yields 
\begin{align}\label{y3+}
  zM_1(z)\Phi'\xpar{y(z)}=M_1(z)+\Oz{\frac12-\eps}.
\end{align}
Hence,
\eqref{Bz2} and \refL{LFFK} yield
\begin{align}\label{y4a}
B(z)\odot\bigsqpar{ zM_1(z)\Phi'\xpar{y(z)}}
&=B(z)\odot M_1(z) +\cc+ \Oz{\frac12-2\eps}  
.\end{align}
This implies, using  \eqref{Bz2}, \eqref{lz1a}, and \refL{LFFK} again,
followed by \eqref{xe},
and recalling 
$\Li_{0,1}\odot L(z)=\Li_{0,1}\odot\Li_{1,0}(z)=\Li_{1,1}(z)$,
\begin{align}\label{y4b}
B(z)\odot\bigsqpar{zM_1(z)\Phi'\xpar{y(z)}}&
=-\gs\qww B(z)\odot L(z)+\cc+\Oz{\gdd-2\eps}  
\notag\\&
=-\gs\qww \bigpar{\Li_{0,1}(z)\odot L(z) -\muu L(z)}+\ccx+\Oz{\gdd-2\eps}  
\notag\\&
=-\gs\qww \Li_{1,1}(z)+\gs\qww\muu L(z)+\ccx+\Oz{\gdd-2\eps}  
.\end{align}
We use the singular expansion of $\Li_{1,1}(z)$: 
\begin{align}\label{li11}
  \Li_{1,1}(z) = \tfrac12 L^2(z)-\gam L(z)+\ccname{\ccii}+\Oz{1-\eps},
\end{align}
which follows from \cite[p.~380]{Flajolet1999}
and is given in \cite[p.~96]{FillK04}, except that the error term there should be $O(|(1 - z) L(z)|)$, not $O(|1 - z|)$. 
%
Consequently, \eqref{y4b} yields
\begin{align}\label{y5}
B(z)\odot\bigsqpar{zM_1(z)\Phi'\xpar{y(z)}}  
=-\tfrac12\gs\qww L^2(z)+\gs\qww(\gam+\muu) L(z)+\cc+\Oz{\gdd-2\eps}  
.\end{align}

For the third term in \eqref{y11}, we have
by \eqref{lz0b}
and \eqref{lz1a},
again using $M_1(z)=\Oz{-\eps}$,
\begin{align}\label{y6}
zM_1(z)^2\Phi''\bigpar{y(z)} 
&= \gss M_1(z)^2 + \Oz{\gdd-2\eps}
\notag\\&
= \gs\qww L^2(z)
-2\frac{{\muu-\psi(-\frac12)}}{\gss} L(z)
+\cc+ \Oz{\gdd-2\eps}.
\end{align}
Finally, \eqref{y11} yields, by
summing \eqref{y12}, \eqref{y5} (twice) and \eqref{y6},
recalling \eqref{psi-1/2},
\begin{align}\label{y7}
  R_2(z) &= 
2\frac{{\gam+\psi(-\frac12)}}{\gss} L(z)
+\cc+ \Oz{\gdd-2\eps}
\notag\\&
=
4(1-\log 2)\gs\qww L(z)
+\ccx+ \Oz{\gdd-2\eps}
.\end{align}
The result \eqref{lz2} now follows by 
\eqref{y7}, \eqref{lh6''}, and \eqref{lgdyy},
and replacing $\eps$ by $\eps/2$ (as we may because $\eps$ is arbitrary).
\end{proof}

This gives the asymptotics for the second moment of the shape functional.
Again, the result includes 
earlier results for special cases in 
\cite{Fill96,MM98,Pittel,FillK04,Caracciolo20}. 
Recall from \eqref{xc} that $F(\ctn)=X_n'(0)-\mu'n$.
\begin{lemma}\label{LZM2}
  Assume \eqref{2+gd} with $\gd>0$.
Then
\begin{align}\label{lzm2}
\E \bigsqpar{\bigpar{X'_n(0)-\mu'n}^2 }
&= 
\E\bigsqpar{F(\ctn)^2} = 
4(1-\log 2)\gs^{-2} n\log n + \On{},
\end{align}
and thus
\begin{align}\label{lzm2v}
\Var X_n'(0)&= 
\Var {F(\ctn)} = 
4(1-\log 2)\gs^{-2} n\log n + \On{}
.\end{align}
\end{lemma}

\begin{proof}
We may assume $\gd\le1$.
The definition \eqref{Mell} and the
singular expansion \eqref{lz2} 
yield by standard  singularity analysis
(using \eqref{liar0}--\eqref{liar20} or \cite[Figure VI.5, p.~388]{FS})
\begin{align}\label{xn2}
q_n\E \bigsqpar{F(\ctn)^2} =
\frac{2^{3/2}(1-\log 2)}{\sqrt\pi}\gs^{-3} n\qqw\log n + \On{-\frac12}
.\end{align}
Hence, \eqref{lzm2} follows by \eqref{qn}.
Finally, \eqref{lzm2v} follows by \eqref{lzm2} and \eqref{xp}.
\end{proof}

\subsection{Higher moments}

We extend the results above to higher moments, using 
the method used earlier for special cases in 
\cite{FillK04,Caracciolo20};
see also \cite{Pittel} for a different method (in another special case).

We prove the following analogue of \cite[Lemma 12.8]{SJ359}.
Note that \eqref{lzc2} is not true for $\ell=1$, since the leading power of
$L(z)$  in that case is $L(z)^1$ by \eqref{lz1a}.
(Also \eqref{lzc2li} fails 
for $\ell=1$ in general.) 

\begin{lemma}\label{LZC2}
Assume \eqref{2+gd} with $0<\gd\le1$.
Then,   for every $\ell\ge2$, 
$M_\ell(z)$ is \gda, and,
for any 
$\eps>0$, 
  \begin{align}\label{lzc2}
 M_\ell(z)&
= 
\gs^{-\ell-1}(1-z)^{(1-\ell)/2}\sum_{j=0}^{\floor{\ell/2}}\kkk_{\ell,j}L(z)^j
+\Oz{-\frac12\ell+\frac12+\gdd-\eps}
\\\label{lzc2li}
&=\gs^{-\ell-1}\sum_{j=0}^{\floor{\ell/2}}\hkkk_{\ell,j}\Li_{(3-\ell)/2,j}(z)
+\Oz{-\frac12\ell+\frac12+\gdd-\eps}
,  \end{align}
for some coefficients $\kkk_{\ell,j}$ and $\hkkk_{\ell,j}$.
The leading coefficients $\kkkx_{2k}:=\kkk_{2k,k}$
 in the case that $\ell = 2 k$ is even
are given by the recursion
\begin{align}\label{kkkx1}
\kkkx_2&
=2^{\xfrac32}(1-\log2),
\\\label{kkkx2}
 \kkkx_{2k}&=2^{-3/2}\sum_{i=1}^{k-1}\binom{2k}{2i}\kkkx_{2i}\kkkx_{2(k-i)},
\qquad k\ge2
.\end{align}
Furthermore,
\begin{align}\label{hkkk}
\hkkk_{2k,k}=\gG\bigpar{k-\tfrac12}\qw\kkk_{2k,k}  
=\gG\bigpar{k-\tfrac12}\qw\kkkx_{2k}  
.\end{align}
\end{lemma}

\begin{proof}
Note first that \eqref{lzc2} and \eqref{lzc2li} are equivalent by
\refL{Liar},
and that \eqref{hkkk} follows using \eqref{liar20}.

We use induction on $\ell$. The base case $\ell=2$ 
(including \eqref{kkkx1})
is \refL{LZ2}, so we
  assume $\ell\ge3$.
We follow the proof of \cite[Lemma 12.8]{SJ359}, 
\emph{mutatis mutandis}.

We first note that $L(z)=\Oz{-\eps}$. Hence, 
for every $\ell'<\ell$,
the induction hypothesis
and (for the case $\ell'=1$) \refL{LZ1} 
show that 
\begin{align}\label{ta0}
  M_{\ell'}(z) = \Oz{-\frac12\ell'+\frac12-\eps}.
\end{align}
(Here and in the sequel we replace 
without further comment,
as we may,
$c\eps$  by $\eps$, for any constant $c$.)
Hence, using \refL{LZ0}, 
for a typical term in \eqref{lh}  (with  $m\ge0$),
\begin{align}\label{lc1qq}
 zM_{\ell_1}(z)\dotsm M_{\ell_m}(z)\Phi\xxm\yz
&= O\bigpar{|1-z|^{-\frac12\sumim \ell_i+\frac{1}{2}m-\eps}\Phi\xxm\yz}
\notag\\
&
=
  \begin{cases}
    O\bigpar{|1-z|^{-\frac12(\ell-\ell_0)+\frac12m-\eps}}, & m\le 2,
\\
  O\bigpar{|1-z|^{-\frac12(\ell-\ell_0)+1+\gdd-\eps}}, & m\ge3.
  \end{cases}
\end{align}
Since $\ell - \ell_0 \ge m$, the exponent here is $<0$.
Hence, \eqref{Bz2} and \refL{LFFK} 
applied $\ell_0$ times yield
\begin{align}\label{lc3}
&B(z)^{\odot\ell_0} \odot 
\bigsqpar{zM_{\ell_1}(z)\dotsm M_{\ell_m}(z)\Phi\xxm\yz}
\notag\\
&\hskip6em
=
  \begin{cases}
    O\bigpar{|1-z|^{-\frac12\ell+\frac12\ell_0+\frac12m-\eps}}, & m\le 2,
\\
  O\bigpar{|1-z|^{-\frac12\ell+\frac12\ell_0+1+\gdd-\eps}}, & m\ge3.
  \end{cases}
\end{align}
If $m=0$, then $\ell_0=\ell\ge3$, and if $m=1$, then $\ell_1<\ell$ and thus
$\ell_0=\ell-\ell_1\ge1$. 
Hence, 
except in the two cases (1) $m=1$ and $\ell_0=1$ and 
(2) $m=2$ and $\ell_0=0$,
we have $m+\ell_0\ge3$, and then
the exponent in \eqref{lc3} is $\ge -\frac12\ell+1+\gdd-\eps$.
Consequently, by \eqref{lh}--\eqref{lh6''},
\begin{align}\label{lc8}
  R_l(z) = \ell B(z) \odot \bigsqpar{z M_{\ell-1}(z) \Phi'\yz}
&+ \frac12\sum_{j=1}^{\ell-1}\binom{\ell}{j} z M_j(z)M_{\ell-j}(z) \Phi''\yz
\notag\\
&+\Oz{-\frac12\ell+1+\gdd-\eps}
.\end{align}
By \eqref{lz0a}, \eqref{ta0}, \eqref{Bz2} and \refL{LFFK}, we have,
similarly to \eqref{y4a},
\begin{align}\label{y4am}
B(z)\odot\bigsqpar{ zM_{\ell-1}(z)\Phi'\xpar{y(z)}}
&=B(z)\odot M_{\ell-1}(z) +\Oz{-\frac12\ell+\frac32-\eps}  
.\end{align}
Hence, using also \eqref{lz0b} and (again) \eqref{ta0}, 
we can simplify \eqref{lc8}  to
\begin{align}\label{lzc8}
  R_l(z) = \ell B(z) \odot M_{\ell-1}(z)
&+ \frac{\gss}2\sum_{j=1}^{\ell-1}\binom{\ell}{j} M_j(z)M_{\ell-j}(z)
+\Oz{-\frac12\ell+1+\gdd-\eps}
.\end{align}

In the remaining estimates we have to be more careful, in particular since
there will be important cancellations.
(This  is as in the case $\ell=2$ treated earlier, but somewhat different.) 

Consider first the Hadamard product in \eqref{lzc8} 
(the case $m=1$ and $\ell_0=1$ above).
We now use the induction hypothesis in the form \eqref{lzc2li} and obtain 
by \eqref{lilix} and \eqref{xe},
using again \eqref{Bz2} and \refL{LFFK} for the error term,
and finally rewriting by \eqref{liar1},
\begin{align}\label{ew2}
&B(z)\odot M_{\ell-1}(z)
\notag\\&\quad
=\gs^{-\ell}
\sum_{j=0}^{\floor{(\ell-1)/2}}\hkkk_{\ell-1,j}
\bigpar{\Li_{(4-\ell)/2,j+1}(z)-\mu'\Li_{(4-\ell)/2,j}(z)}
+\Oz{-\frac12\ell+1+\gdd-\eps}
\notag\\&\quad
=\gs^{-\ell}
\sum_{k=0}^{\floor{(\ell+1)/2}}\ci_{\ell,k}\Li_{(4-\ell)/2,k}(z)
+\Oz{-\frac12\ell+1+\gdd-\eps}
\notag\\&\quad
= \gs^{-\ell}  (1-z)^{-\frac12\ell+1} \sum_{k=0}^{\floor{(\ell+1)/2}} \cii_{\ell,k} L(z)^k
+\Oz{-\frac12\ell+1+\gdd-\eps}
,\end{align}
where the leading coefficient in the sum
is, using \eqref{liar2} and \eqref{liar20},
\begin{align}\label{ew3}
\cii_{\ell,\floor{(\ell+1)/2}}
=
\gG\xpar{\ell/2-1}
\ci_{\ell,\floor{(\ell+1)/2}}
=
\gG\xpar{\ell/2-1}
\hkkk_{\ell-1,\floor{(\ell-1)/2}}
=\kkk_{\ell-1,\floor{(\ell-1)/2}}
.\end{align}
The leading term in \eqref{ew2} is  thus
\begin{align}\label{ew4}
\gs^{-\ell}\kkk_{\ell-1,\floor{(\ell-1)/2}}
(1-z)^{-\frac12\ell+1} L(z)^{\floor{(\ell+1)/2}}
.\end{align}

Consider now the  terms with $j=1$ and $j=\ell-1$ in the sum in
\eqref{lzc8}.
By \refL{LZ1} and the induction hypothesis, we  have
\begin{align}\label{ew5}
\gss M_1(z)M_{\ell-1}(z)&
=\gs^{-\ell} (1-z)^{-\frac12\ell+1} \sum_{j=0}^{\floor{(\ell-1)/2}}\kkk_{\ell-1,j}
\bigsqpar{-L(z)^{j+1}+cL(z)^j}
\notag\\&\hskip2em
+\Oz{-\frac12\ell+1+\gdd-\eps}
.\end{align}
Note that the leading term in \eqref{ew5} cancels \eqref{ew4}.
Consequently, \eqref{ew2}--\eqref{ew5} yield 
\begin{align}\label{ew6}
&\ell B(z) \odot  M_{\ell-1}(z)
+ \frac{\gss}2\cdot2\cdot\binom{\ell}{1} M_1(z)M_{\ell-1}(z)
\notag\\&\hskip2em
=  (1-z)^{-\frac12\ell+1} \sum_{k=0}^{\floor{(\ell-1)/2}} \ciii_{\ell,k}L(z)^k
+\Oz{-\frac12\ell+1+\gdd-\eps}
.\end{align}

The remaining terms in \eqref{lzc8} yield immediately, by the induction
hypothesis,
\begin{align}\label{ew7}
&\frac{\gss}2\sum_{j=2}^{\ell-2}\binom{\ell}{j} M_j(z)M_{\ell-j}(z)
=  (1-z)^{-\frac12\ell+1} \sum_{k=0}^{\floor{\ell/2}} \civ_{\ell,k}L(z)^k
+\Oz{-\frac12\ell+1+\gdd-\eps}
.\end{align}
Finally, \eqref{lzc8} and \eqref{ew6}--\eqref{ew7} yield
  \begin{align}\label{ewr}
R_\ell(z) = 
(1-z)^{-\frac12\ell+1} \sum_{j=0}^{\floor{\ell/2}}\cv_{\ell,j}L(z)^j
+\Oz{-\frac12\ell+1+\gdd-\eps}
,  \end{align}
and \eqref{lzc2} follows by \eqref{lh6''} and \eqref{lgdyy}, 
which completes the induction step. 

It remains only to show the recursion \eqref{kkkx2} for the leading
coefficients.
If $\ell=2k$ is even, with $\ell\ge4$, then \eqref{ew6} does not contribute
to $\cv_{2k,k}$ 
 nor 
thus to $\kkk_{2k,k}$, 
and neither do the terms in \eqref{ew7} with $j$ odd.
Hence, the argument above  yields
\begin{align}\label{eiv}
 \cv_{2k,k}=\frac12\sum_{i=1}^{k-1}\binom{2k}{2i}\gs^{-2k}\kkk_{2i,i}\kkk_{2k-2i,k-i}
\end{align}
and thus, recalling again \eqref{lgdyy},
\begin{align}\label{ewk}
 \kkk_{2k,k}=2^{-3/2}\sum_{i=1}^{k-1}\binom{2k}{2i}\kkk_{2i,i}\kkk_{2k-2i,k-i},
\end{align}
which is \eqref{kkkx2}.
\end{proof}

The recursion \eqref{kkkx2} is the same as \cite[(C.35)]{SJ359},
and thus has the same solution \cite[(C.40)]{SJ359}, \ie,
\begin{align}\label{ick8z}
  \kkkx_{2k}=2^{3/2}\frac{(2k)!\,(2k-2)!}{(k-1)!\,k!}
  d_1^k,
\qquad k\ge1,
\end{align}
with, see \cite[(C.36)]{SJ359} and \eqref{kkkx1},
\begin{align}\label{zd1}
  d_1:=2^{-3/2}\kkkx_2/2  =\tfrac12(1-\log 2)
.\end{align}

This is what we need to complete the proof of the asymptotic normality
 of $F(\ctn)$.

\begin{proof}[Proof of \refT{TZ}]
If $\ell\ge2$, then
\eqref{Mell},
the  expansion \eqref{lzc2li}, \eqref{lixr}, and standard  singularity analysis
yield
\begin{align}\label{xnl}
q_n\E \bigsqpar{F(\ctn)^\ell} =\gs^{-\ell-1}
\hkkk_{\ell,\floor{\ell/2}}
n^{(\ell-3)/2}(\log n)^{\floor{\ell/2}}
+O\bigpar{n^{(\ell-3)/2}(\log n)^{\floor{\ell/2}-1}}
.\end{align}
Hence, using \eqref{qn},
\begin{align}\label{gnu}
\E \bigsqpar{F(\ctn)^\ell} 
=\gs^{-\ell}
\sqrt{2\pi}\hkkk_{\ell,\floor{\ell/2}}
n^{\ell/2}(\log n)^{\floor{\ell/2}}
+O\bigpar{n^{\ell/2}(\log n)^{\floor{\ell/2}-1}}
.\end{align}
Consequently,
\begin{align}\label{tussi}
\frac{\E \bigsqpar{F(\ctn)^\ell} }{(n\log n)^{\ell/2}}
\to
  \begin{cases}
    0,& \ell=2k+1\ge3,
\\
\gs^{-2k}
\sqrt{2\pi}\hkkk_{2k,k},
&\ell=2k\ge2.
  \end{cases}
\end{align}
Furthermore, \eqref{tussi} holds also for $\ell=1$ (with limit 0) by
\eqref{xp}.

For even $\ell=2k$, the limit in \eqref{tussi} is  by 
\eqref{hkkk}, \eqref{ick8z}, and \eqref{zd1}, 
\cf{} \cite[(C.41)]{SJ359},
\begin{align}\label{april}
\gs^{-2k}\frac{\sqrt{2\pi}}{\gG(k-\frac12)}\kkkx_{2k}
&=
\gs^{-2k}\frac{4\sqrt{\pi}}{\gG(k-\frac12)}
\frac{(2k)!\,(2k-2)!}{(k-1)!\,k!} d_1^k
=\gs^{-2k}2^{2k}\frac{(2k)!}{k!}d_1^k
\notag\\&
=\bigpar{8d_1\gs\qww}^k\cdot(2k-1)!! 
=
\bigpar{4(1-\log2)\gs\qww}^k
\cdot
(2k-1)!!
.\end{align}
Consequently, the limits appearing
in \eqref{tussi} are the moments of a normal
distribution $N\bigpar{0,4(1-\log2)\gs\qww}$, and thus \eqref{tz} follows by
the method of moments. (Recall that $F(\ctn)=X_n'(0)-\mu'n$ by \eqref{xc}.)
\end{proof}

\section{Imaginary powers}\label{S:I}
\ccreset

In this section, 
we consider $X_n(\ga)$ in \eqref{Xn} when
the exponent $\ga$ is purely imaginary, \ie, $\Re\ga=0$.
We exclude the trivial case $\ga=0$, when $X_n(\ga)=n$ is non-random.
We assume throughout the section 
that $0<\gd<1$ and that \eqref{2+gd} holds.
As above, $\eps$ is an arbitrarily small positive number, and we replace
$c\eps$ by $\eps$ without  comment.

We follow rather closely the argument for the case $0<\rea<1/2$
in \cite[\S12.4--6]{SJ359}, but we will see new terms appearing that will
lead to the dominating terms with logarithmic factors for the moments; 
this is very similar to the argument in
\refS{S:shape}, but we will see some differences.
(Notably, there are no  cancellations of leading terms  like those in \refS{S:shape}.)

As in \cite[\S12.4]{SJ359}, we define 
\begin{align}
  b_n:=n^\ga-\mu(\ga),
\end{align}
with the  following generating function
(\cf{} \cite[(12.44)]{SJ359}
and \eqref{lia},
and note $\Li_0(z)=z(1-z)\qw$):
\begin{align}\label{Boxli}
B(z)&= B_\ga(z):=\sumn b_nz^n
=\Li_{-\ga}(z)-\mu(\ga)\Li_0(z)
\\&\label{Box2}
= \gG(1+\ga)(1-z)^{-\ga-1}-\mu(\ga)(1-z)\qw+O(1)
\\&\label{Box3}
= \Oz{-1}
.\end{align}

Let now 
$\dF(T)=\dF_\ga(T)$ denote the additive functional defined by the toll function
$\df_\ga(T):=b_{|T|}$.
Thus, 
\begin{align}\label{FX}
 \dF_\ga(\tn)=X_n(\ga)-n\mu(\ga). 
\end{align}

\subsection{The mean}\label{S:Imean}

For the mean, we define the generating function
\begin{align}\label{M1}
M_\ga(z):=\E \bigsqpar{\dF_\ga(\cT) z^{|\cT|}}
=\sumn q_n \E[\dF_\ga(\tn)] z^n
.\end{align}
We then have,
as in \eqref{xsx1} and \cite[(12.29)]{SJ359},
\begin{align}\label{sx1}
  M_\ga(z)
=\frac{zy'(z)}{y(z)}\cdot\bigpar{ B_\ga(z)\odot y(z)}
.\end{align}
Thus $M_\ga(z)$ is \gda. Further,
we have by \eqref{lili}, \eqref{Boxli}, and \eqref{yli}, 
using \eqref{Box3} and \refL{LFFK} for the error term in \eqref{yli},
and then using for the second line \eqref{lia} and $\gG(-\frac12)=-2\sqrt\pi$,
\begin{align}\label{sixten2}
  B_\ga(z)&\odot y(z )
=  
\frac{1}{\sqrt{2\pi}\gs}\Li_{3/2-\ga}(z)
-\frac{\mu(\ga)}{\sqrt{2\pi}\gs}\Li_{3/2}(z)
+\cc+\Oz{\frac12+\frac\gd2}
\notag\\&
= 
\frac{\gG(\ga-\half)}{\sqrt{2\pi}\gs}
(1-z)^{\half-\ga}
+2\qq\gs\qw\mu(\ga)(1-z)\qq
+\cc+\Oz{\frac12+\frac\gd2}.
\end{align}
 Further,
similarly to \eqref{xh},
\begin{align}\label{xhga}
  (B_\ga \odot y)(1) = \sumn b_n q_n 
=\sumn q_n [n^\ga-\mu(\ga)]
=\E|\cT|^\ga- \mu(\ga)=0.
\end{align}
Thus, 
letting $z\to1$ in \eqref{sixten2}
shows that $\ccx=(B_\ga\odot y)(1)=0$.

Finally,  \eqref{sx1}, \eqref{lgdyy}, and \eqref{sixten2} yield,
using \eqref{lia} again,
\begin{align}\label{theo}
M_\ga(z)&=
\frac{\gG(\ga-\half)}{2\sqrt\pi\gss}(1-z)^{-\ga}
+\gs\qww\mu(\ga)
+\Oz{\frac\gd2}
\\&\label{theoli}
=\frac{\gG(\ga-\half)}{2\sqrt\pi\gss\gG(\ga)}\Li_{1-\ga}(z)
+c
+\Oz{\frac\gd2}
.\end{align}
 
Singularity analysis now yields, 
from \eqref{M1}  and 
\eqref{theoli},
\begin{align}\label{ixo}
  q_n \E[\dF_\ga(\tn)] = 
\frac{\gG(\ga-\half)}{2\sqrt\pi\gss\gG(\ga)}n^{\ga-1}
+O\bigpar{n^{-1-\gdii}}
\end{align}
and thus, by \eqref{qn},
\begin{align}\label{ixp}
 \E[\dF_\ga(\tn)] = 
\frac{\gG(\ga-\half)}{\sqrt2\gs\gG(\ga)}n^{\frac12+\ga}
+\On{\frac12-\gdii}
\end{align}
Hence, recalling \eqref{FX},
\begin{align}\label{ixq}
  \E X_n(\ga)=\mu(\ga)n 
+\frac{\gG(\ga-\half)}{\sqrt2\gs\gG(\ga)}n^{\frac12+\ga}
+\On{\frac12-\gdii}
.\end{align}
This agrees with \cite[Theorem 1.7(ii)]{SJ359}
(proved without \eqref{2+gd}, and by different methods), except that the
error  estimate here is smaller.

\subsection{Higher moments}\label{S:Imom}

For higher moments, we need mixed moments for 
$\ga$ and $\bga=-\ga$. 
Thus, somewhat more generally, fix $\ga_1$ and $\ga_2$ with
$\rea_1=\rea_2=0$
but $\ga_1\neq0\neq\ga_2$.
We define, for integers $\ell_1,\ell_2\ge0$,
the generating function
  \begin{align}
\label{Mell2}
  M_{\ell_1,\ell_2}(z)
&:=\E \bigsqpar{\dF_{\ga_1}(\cT)^{\ell_1}\dF_{\ga_2}(\cT)^{\ell_2} z^{|\cT|}}
=\sumn q_n 
\E\bigsqpar{\dF_{\ga_1}(\ctn)^{\ell_1}\dF_{\ga_2}(\ctn)^{\ell_2}} z^n
.\end{align}
Thus $M_{1,0}=M_{\ga_1}$ and $M_{0,1}=M_{\ga_2}$ are given by \eqref{sx1}.
The  functions $M_{\ell,r}$
can then be found by the following recursion, given in \cite[(12.75)]{SJ359},
for every $\ell,r\ge0$
with 
{$\ell+r\ge1$}: 
\begin{multline}\label{lhz}
  M_{\ell,r}(z)
=
\frac{z y'(z)}{y(z)}
\sum_{m=0}^{\ell+r} \frac{1}{m!}\sumxx \binom{\ell}{\ell_0,\dots,\ell_m}
\binom{r}{r_0,\dots,r_m}
B_{\ga_1}(z)^{\odot\ell_0} 
\\
\odot 
B_{\ga_2}(z)^{\odot r_0} \odot 
\bigsqpar{zM_{\ell_1,r_1}(z)\dotsm M_{\ell_m,r_m}(z)\Phi\xxm\yz},
\end{multline}
where $\sumxx$ is the sum over all pairs of
$(m+1)$-tuples $(\ell_0,\dots,\ell_m)$  and $(r_0,\dots,r_m)$ 
of non-negative integers that sum to $\ell$ and $r$, respectively,
such that $1\le\ell_i+r_i<\ell + r$ for every $i\ge1$.
(Note that there are two typos in \cite{SJ359}:
the lower summation limit should be $m=0$, and  
the final qualification ``$i\ge1$'' is missing there.)
It follows by induction that every $M_{\ell,r}$ is \gda.

We define for convenience $R_{\ell,r}(z)$ as the sum in \eqref{lhz}; thus 
\begin{align}\label{mr}
    M_{\ell,r}(z)
=
\frac{z y'(z)}{y(z)}R_{\ell,r}(z).
\end{align}

Let us first consider second moments. Taking $\ell=r=1$ in \eqref{lhz}
yields, recalling \eqref{Phiy},
\begin{align}\label{the1}
    R_{1,1}(z)&
=
B_{\ga_1}(z)\odot B_{\ga_2}(z)\odot y(z)
+ B_{\ga_1}(z)\odot [zM_{0,1}(z)\Phi'(y(z))]
\notag\\&\qquad
+ B_{\ga_2}(z)\odot [zM_{1,0}(z)\Phi'(y(z))]
+ zM_{1,0}(z)M_{0,1}(z)\Phi''(y(z)).
\end{align}
The first term  is, 
by \eqref{Box3} and \eqref{sixten2} together with \refL{LFFK},
\begin{align}
 \Oz{-1}\odot\Oz{1/2}=\cc+\Oz{1/2}. 
\end{align}
For the other terms in \eqref{the1}, we first 
note from \eqref{theo} that  $M_{1,0}(z)=M_{\ga_1}(z)=O(1)$ and
 $M_{0,1}(z)=M_{\ga_2}(z)=O(1)$. Thus, using also
\eqref{lz0a}--\eqref{lz0b}, \eqref{Box3} and \refL{LFFK}, we may simplify 
to
\begin{align}\label{the2}
    R_{1,1}(z)&
=\cc
+ B_{\ga_1}(z)\odot M_{0,1}(z)
+ B_{\ga_2}(z)\odot M_{1,0}(z)
+ M_{1,0}(z)M_{0,1}(z)\gss
\notag\\&\qquad
+\Oz{\gd/2}.
\end{align}
Furthermore, \eqref{theo}  yields
\begin{align}\label{MM}
  M_{1,0}(z)M_{0,1}(z)&
=\cc(1-z)^{-\ga_1}+\cc(1-z)^{-\ga_2}+\ccname{\ccab}(1-z)^{-\ga_1-\ga_2}+\cc
+\Oz{\frac{\gd}2}
.\end{align}
We compute the Hadamard products in \eqref{the2} by
\eqref{lili}, \eqref{Boxli} and \eqref{theoli}, 
using again \eqref{Box3} and \refL{LFFK} for the error term.
Together with \eqref{MM}, this yields from \eqref{the2}
a result that we write, using \eqref{lia}, as
\begin{align}\label{the3}
    R_{1,1}(z)&
= 
\Bigpar{\frac{\gG(\ga_2-\half)}{2\sqrt\pi\gss\gG(\ga_2)}
+\frac{\gG(\ga_1-\half)}{2\sqrt\pi\gss\gG(\ga_1)}}
\Li_{1-\ga_1-\ga_2}(z)
\notag\\&\quad
+\ccname\cceri(1-z)^{-\ga_1}
+\ccname\ccmag(1-z)^{-\ga_2}
+\ccab(1-z)^{-\ga_1-\ga_2}
+\cc
\notag\\&\quad
+\Oz{\gd/2}.
\end{align}
If $\ga_1+\ga_2\neq0$, we use \eqref{lia} also on the first term and obtain
\begin{align}\label{the4}
    R_{1,1}(z)&
= 
\cc(1-z)^{-\ga_1-\ga_2}
+\cceri(1-z)^{-\ga_1}
+\ccmag(1-z)^{-\ga_2}
+\cc
\notag\\&\quad
+\Oz{\gd/2}.
\end{align}
On the other hand, if $\ga_1+\ga_2=0$, we recall that $\Li_1(z)=L(z)$,
and thus \eqref{the3} yields
\begin{align}\label{the0}
    R_{1,1}(z)&
= 
\frac{1}{\sqrt\pi\gss}\Re\frac{\gG(\ga_1-\half)}{\gG(\ga_1)}\cdot
L(z)
+\cceri(1-z)^{-\ga_1}
+\ccmag(1-z)^{-\ga_2}
+\cc
\notag\\&\quad
+\Oz{\gd/2}.
\end{align}

We can now obtain $M_{1,1}(z)$ from \eqref{the4}--\eqref{the0} by \eqref{mr}
and \eqref{lgdyy}.
We do not state the result separately, but proceed immediately to a general
formula.

\begin{lemma}\label{Lth}
  Let $\ga\neq0$ with $\rea=0$, and take 
$\ga_1=\ga$ and $\ga_2=\bga=-\ga$.
Then, for each pair of integers $\ell,r\ge0$ with $\ell+r\ge2$,
$M_{\ell,r}(z)$ is \gda{} and
we have, 
for some coefficients $\kk_{\ell,r;j,k}$ and $\hkk_{\ell,r;j,k}$, and 
every  $\eps>0$,
\begin{align}
\label{lth}
  M_{\ell,r}(z)
&=
\sum_{j,k}\kk_{\ell,r;j,k} (1-z)^{(1-\ell-r)/2+j\ga} L(z)^k
+\Oz{\frac12(1-\ell-r)+\gdde}
\\\label{lthli}
&=
\sum_{j,k}\hkk_{\ell,r;j,k}\Li_{(3-\ell-r)/2+j\ga,k}(z)
+\Oz{\frac12(1-\ell-r)+\gdde}
,\end{align}
where the sums are over integers $j$ and $k$ with $-\ell\le j\le r$ and
$0\le k\le \ell\land r$.

Furthermore, if $\ell+r=1$, then \eqref{lth} holds (but not \eqref{lthli}).

If\/ $\ell=r$, then the only non-zero coefficients with $k=\ell=r$ are
\begin{align}\label{ltha}
  \kk_{\ell,\ell;0,\ell}&=\gs^{-2\ell-1}\kkx_\ell,
\\\label{lthb}
  \hkk_{\ell,\ell;0,\ell}&= \gG\bigpar{\ell-\tfrac12}\qw \kk_{\ell,\ell;0,\ell}
=\frac{\gs^{-2\ell-1}}{\gG\bigpar{\ell-\frac12}}\kkx_\ell,
\end{align}
where $\kkx_\ell$ is given by the recursion
\begin{align}\label{lthkk1}
  \kkx_1&=
\frac{1}{\sqrt{2\pi}}\Re\frac{\gG(\ga-\half)}{\gG(\ga)},
\\\label{lthkk}
  \kkx_\ell&
=2^{-3/2}\sum_{i=1}^{\ell-1}\binom{\ell}{i}^2\kkx_{i}\kkx_{\ell-i},
\qquad \ell\ge2.
\end{align}
\end{lemma}

\begin{proof}
Note first that for $\ell+r=1$, \eqref{lth} follows from \eqref{theo}.
(We see also from \eqref{theoli} that \eqref{lthli} would hold if we add a
constant term; the problem is that $\Li_1(z)$ is $L(z)$ and not a constant.)

Assume in the rest of the proof that $\ell+r\ge2$.
Then the expansions \eqref{lth} and \eqref{lthli} are
equivalent by \refL{Liar}; furthermore, for the leading terms,
\eqref{ltha} and \eqref{lthb} are equivalent by \eqref{liar20}.

Consider next the case $\ell+r=2$.
If $(\ell,r)=(2,0)$ or $(0,2)$, we use \eqref{the4} with
$\ga_1=\ga_2=\pm\ga$ and obtain \eqref{lth} by \eqref{mr} and
\eqref{lgdyy}.
(Now only terms with $k=0$  appear.)

If $\ell=r=1$, we similarly use \eqref{the0}, \eqref{mr} and \eqref{lgdyy}
and obtain \eqref{lth} including a single term with $k=1$, viz.\
$\kk_{1,1;0,1}L(z)(1-z)^{-1/2}$ with $\kk_{1,1;0,1}$ given by \eqref{ltha} and 
\eqref{lthkk1}.

For $\ell+r\ge3$, we use induction on $\ell+r$.
By the induction hypothesis \eqref{lth} (including the case $\ell+r=1$
just proved by \eqref{theo}), we have for every $(\ell',r')$ with 
$1\le\ell'+r'<\ell+r$, 
\begin{align}\label{ita0}
  M_{\ell',r'}(z) = \Oz{-\frac12(\ell'+r')+\frac12-\eps}.
\end{align}
Consequently, 
for a typical term in \eqref{lhz},
as in \eqref{lc1qq} and
using again \refL{LZ0}, 
\begin{align}\label{ilc1qq}
& zM_{\ell_1,r_1}(z)\dotsm M_{\ell_m,r_m}(z)\Phi\xxm\yz
= O\bigpar{|1-z|^{-\frac12\sumim (\ell_i+r_i)+\frac{1}{2}m-\eps}\Phi\xxm\yz}
\notag\\
&\hskip4em
=
  \begin{cases}
    O\bigpar{|1-z|^{-\frac12(\ell+r-\ell_0-r_0)+\frac12m-\eps}}, & m\le 2,
\\
  O\bigpar{|1-z|^{-\frac12(\ell+r-\ell_0-r_0)+1+\gdd-\eps}}, & m\ge3.
  \end{cases}
\end{align}
Again the exponent here is $<0$, and it follows by
 \eqref{Box3} and \refL{LFFK} 
that
\begin{align}\label{ilc3}
B_{\ga_1}(z)^{\odot\ell_0} \odot 
B_{\ga_2}(z)^{\odot r_0} \odot 
\bigsqpar{zM_{\ell_1,r_1}(z)\dotsm M_{\ell_m,r_m}(z)\Phi\xxm\yz}
\notag\\
=
  \begin{cases}
    O\bigpar{|1-z|^{-\frac12(\ell+r)+\frac12(\ell_0+r_0)+\frac12m-\eps}}, & m\le 2,
\\
  O\bigpar{|1-z|^{-\frac12(\ell+r)+\frac12(\ell_0+r_0)+1+\gdd-\eps}}, & m\ge3.
  \end{cases}
\end{align}
As in the proof of \refL{LZC2},
except in the two cases (1) $m=1$ and $\ell_0+r_0=1$ and 
(2) $m=2$ and $\ell_0=r_0=0$ 
we have $m+\ell_0+r_0\ge3$, and then
the exponent in \eqref{ilc3} is $\ge -\frac12(\ell+r)+1+\gdd-\eps$.
Consequently, by \eqref{lhz}--\eqref{mr},
\begin{samepage}
\begin{align}\label{ilc8}
  R_{l,r}(z) &= 
\ell B_{\ga_1}(z) \odot \bigsqpar{z M_{\ell-1,r}(z) \Phi'\yz}
+r B_{\ga_2}(z) \odot \bigsqpar{z M_{\ell,r-1}(z) \Phi'\yz}
\notag\\&\qquad
+ \frac12\sumsum_{0<i+j<\ell+r}
\binom{\ell}{i}\binom{r}{j} z M_{i,j}(z)M_{\ell-i,r-j}(z) \Phi''\yz
\notag\\&\qquad
+\Oz{-\frac12(\ell+r)+1+\gdd-\eps}
.\end{align}
\end{samepage}%
As in \eqref{lc8}--\eqref{lzc8} and \eqref{the1}--\eqref{the2}, this can be
simplified,
using \eqref{lz0a}--\eqref{lz0b}, \eqref{ita0}, \eqref{Box3} and \refL{LFFK}, 
and we obtain
\begin{samepage}
\begin{align}\label{ilzc8}
  R_{l,r}(z) &= 
\ell B_{\ga_1}(z) \odot  M_{\ell-1,r}(z)
+r B_{\ga_2}(z) \odot  M_{\ell,r-1}(z) 
\notag\\&\qquad
+ \frac{\gss}2\sumsum_{0<i+j<\ell+r}
\binom{\ell}{i}\binom{r}{j} M_{i,j}(z)M_{\ell-i,r-j}(z)
+\Oz{-\frac12(\ell+r)+1+\gdd-\eps}.
\end{align}
\end{samepage}%
By the induction hypothesis in the form \eqref{lthli} and \eqref{Boxli}, 
using as always \refL{LFFK} for the error term,
we have
\begin{align}
 B_{\ga_1}(z) \odot  M_{\ell-1,r}(z)
&=
\sum_{j,k}\hkk_{\ell-1,r;j,k}
\Li_{(4-\ell-r)/2+j\ga,k}(z)
\odot\bigpar{\Li_{-\ga}(z)-\mu(\ga)\Li_0(z)}
\notag\\&\hskip4em
+\Oz{-\frac12(\ell+r)+1+\gdd-\eps}
\end{align}
summing over $-(\ell-1)\le j\le r$ and $0\le k\le (\ell-1)\land r$.
By \eqref{lilix}, this can be rearranged as
\begin{align}\label{win1}
  \sum_{j,k}\ci_{\ell,r;j,k}
\Li_{(4-\ell-r)/2+j\ga,k}(z)
+\Oz{-\frac12(\ell+r)+1+\gdd-\eps},
\end{align}
now summing over 
$-\ell\le j\le r$ and $0\le k\le (\ell-1)\land r$.
By \refL{Liar}, this can also be  written
\begin{align}\label{win2}
  \sum_{j,k}\cii_{\ell,r;j,k}
(1-z)^{(2-\ell-r)/2+j\ga} L(z)^k
+\Oz{-\frac12(\ell+r)+1+\gdd-\eps},
\end{align}
still summing over $-\ell\le j\le r$ and $0\le k\le (\ell-1)\land r$.

By symmetry, $B_{\ga_2}(z) \odot  M_{\ell,r-1}(z) $ can also be written as 
\eqref{win2} (with different coefficients $\cii_{\ell,r;j,k}$), 
now summing over
$-\ell\le j\le r$ and $0\le k\le \ell\land (r-1)$.

Finally, the double sum in \eqref{ilzc8} can by the induction hypothesis
\eqref{lth} also be written as \eqref{win2}, summing over 
$-\ell\le j\le r$ and $0\le k\le \ell\land r$.

Consequently, \eqref{ilzc8} yields 
\begin{align}\label{win}
R_{\ell,r}(z) =  
\sum_{j,k}\ciii_{\ell,r;j,k}
(1-z)^{(2-\ell-r)/2+j\ga} L(z)^k
+\Oz{-\frac12(\ell+r)+1+\gdd-\eps},
\end{align}
summing over $-\ell\le j\le r$ and $0\le k\le \ell\land r$.
By \eqref{mr} and \eqref{lgdyy}, this implies \eqref{lth}, which completes
the induction proof of \eqref{lth}--\eqref{lthli}.

Now consider the case $\ell=r\ge2$.
 We see that then the only terms above with $k=\ell=r$
come from the double sum in \eqref{ilzc8}; moreover, they appear only for
terms there with $i=j$, and we obtain by induction
that the only non-zero coefficient 
in \eqref{win} with $k=\ell$ is, using \eqref{ltha},
\begin{align}\label{winc}
  \ciii_{\ell,\ell;0,\ell}
=\frac{\gss}2\sum_{i=1}^{\ell-1}\binom{\ell}{i}^2\kk_{i,i;0,i}\kk_{\ell-i,\ell-i;0,\ell-i}
=\frac12\gs^{-2\ell}\sum_{i=1}^{\ell-1}\binom{\ell}{i}^2\kkx_{i}\kkx_{\ell-i}
\end{align}
Hence, when deriving \eqref{lth} from \eqref{win} by \eqref{mr} and
\eqref{lgdyy},  we also find that the only non-zero coefficient with
$k=\ell$ is
\begin{align}\label{winkk}
  \kk_{\ell,\ell;0,\ell}
=2\qqw\gs\qw  \ciii_{\ell,\ell;0,\ell}
=2^{-3/2}\gs^{-2\ell-1}\sum_{i=1}^{\ell-1}\binom{\ell}{i}^2\kkx_{i}\kkx_{\ell-i}
.\end{align}
This proves \eqref{ltha} and \eqref{lthkk}.
\end{proof}

The recursion \eqref{lthkk} is the same as \cite[(D.6)]{SJ359}, and thus has
the same solution \cite[(D.10)]{SJ359}
\begin{align}\label{alla}
\kkx_\ell
= 2^{3/2} \frac{\ell!\,(2\ell-2)!}{(\ell-1)!}d_1^\ell,
\end{align}
with,
by \cite[(D.9)]{SJ359}
and \eqref{lthkk1},
\begin{align}\label{id1}
d_1&:=2^{-3/2}\kkx_{1}
=\frac{1}{4\sqrt{\pi}}\Re\frac{\gG(\ga-\half)}{\gG(\ga)}
.\end{align}

\begin{proof}[Proof of \refT{Tit}]
We have $\ga=\ii t$.
If $\ell+r\ge2$, then
\eqref{Mell2},
 \eqref{lthli}, \eqref{lixr}, and  singularity analysis
yield
\begin{align}\label{ixnl}
q_n\E \bigsqpar{\dF_{\ga}(\ctn)\strut^\ell\,\overline{\dF_{\ga}(\ctn)}^r} =
q_n\E \bigsqpar{\dF_{\ga}(\ctn)^\ell\dF_{\bga}(\ctn)^r} =
O\bigpar{
n^{(\ell+r-3)/2}(\log n)^{\ell\land r}}
.\end{align}
When $\ell=r$, we find more  precisely
\begin{align}\label{ixnl=}
q_n\E \bigsqpar{\dF_{\ga}(\ctn)\strut^\ell\,\overline{\dF_{\ga}(\ctn)}^\ell} 
=
\hkk_{\ell,\ell;0,\ell}
n^{(2\ell-3)/2}(\log n)^{\ell}
+O\bigpar{n^{(2 \ell-3)/2}(\log n)^{\ell-1}}
.\end{align}
Hence, using \eqref{qn} and \eqref{lthb},
\begin{align}\label{ignu}
\E \bigsqpar{\dF_{\ga}(\ctn)\strut^\ell\,\overline{\dF_{\ga}(\ctn)}^r} =
  \begin{cases}
    O\bigpar{n^{(\ell+r)/2}(\log n)^{\ell\land r}},
& \ell\neq r,
\\
\gs^{-2\ell}\frac{\sqrt{2\pi}}{\gG(l-\frac12)}\kkx_\ell
n^{\ell}(\log n)^{\ell}
+ O\bigpar{n^{\ell}(\log n)^{\ell-1}},
& \ell=r.
  \end{cases}
\end{align}
Consequently,
\begin{align}\label{itussi}
\frac{\E \bigsqpar{\dF_{\ga}(\ctn)\strut^\ell\,\overline{\dF_{\ga}(\ctn)}^r}}
{(n\log n)^{(\ell+r)/2}}
\to
  \begin{cases}
    0,& \ell\neq r,
\\
\gs^{-2\ell}\frac{\sqrt{2\pi}}{\gG(l-\frac12)}\kkx_\ell,
&\ell=r\ge1.
  \end{cases}
\end{align}
Furthermore, \eqref{itussi} holds also for $\ell+r=1$ by
\eqref{ixp}.

For $\ell=r$, the limit in \eqref{itussi} is by 
 \eqref{alla} and \eqref{id1}, 
\cf{} \cite[(D.11)]{SJ359},
\begin{align}\label{nov}
\gs^{-2\ell}\frac{\sqrt{2\pi}}{\gG(\ell-\frac12)}\kkx_{\ell}
&=
\gs^{-2\ell}\frac{4\sqrt{\pi}}{\gG(\ell-\frac12)}
\frac{\ell!\,(2\ell-2)!}{(\ell-1)!} d_1^\ell
=\gs^{-2\ell}2^{2\ell}\ell!\,d_1^\ell
\notag\\&
=\bigpar{4d_1\gs\qww}^\ell\cdot \ell! 
=
\Bigpar{
\frac{1}{\sqrt{\pi}\gss}\Re\frac{\gG(\ga-\half)}{\gG(\ga)}}^\ell
\cdot
\ell!
.\end{align}
Consequently, by \eqref{cnormal},
the limits in \eqref{itussi} are the moments of a symmetric
complex normal
distribution with variance \eqref{tit2}, and thus \eqref{tit} follows by
the method of moments. (Recall that $\dF_\ga(\ctn)=X_n(\ga)-\mu(\ga)n$ by
\eqref{FX}.) 

Finally, the claim in \eqref{tit2} that the variance is nonzero
follows from the same claim in \eqref{ri2} (where the variance is the same
up to a factor  $\gs^2 / 2$), which is shown in \cite[Theorem D.1]{SJ359}, with
the correction in 
\refApp{Atypo} below.
\end{proof}

\subsection{Joint distributions}\label{S:Ijoint}

We can extend the arguments above to joint distributions of several
$X_n(\ga)$ with different imaginary $\ga$.
Since we have $X_n(\bga)=\overline{X_n(\ga)}$, it suffices to consider the
case $\Im\ga>0$. In this case, different $X_n(\ga)$ are asymptotically
independent, as is stated more precisely in the following theorem.

\begin{theorem}\label{Tjoint}
  For any finite set\/ $t_1,\dots,t_n$ of distinct positive numbers,
the complex random variables 
$\bigpar{X_n(\ii t_k)-\mu(\ii t_k)n}/\sqrt{n\log n}$
converge, as \ntoo, jointly in distribution to
independent symmetric complex normal variables $\zeta_{\ii t_k}$ with variances
given by \eqref{tit2}.
\end{theorem}

This can be interpreted as joint convergence 
(in the product topology)
of the entire family
$\set{X_n(\ii t):t>0}$ of random variables, after normalization, 
to an (uncountable) family of
\emph{independent} symmetric complex normal variables $\zetat$.
As said in \refR{Rjoint}, this behaviour is strikingly different from the
cases $\rea<0$ and $\rea>0$, where we have joint convergence to analytic
random functions of $\ga$.

\begin{proof}
  We argue as above, using the method of moments and singularity analysis of
  generating  functions, with mainly notational differences.
We give only a sketch, leaving further details to the reader.

For a sequence of arbitrary non-zero imaginary numbers $\ga_1,\dots,\ga_\ell$
(allowing repetitions), define the generating function
  \begin{align}
\label{Mellj}
  M_{\ga_1,\dots,\ga_\ell}(z)
&:=\E \bigsqpar{\dF_{\ga_1}(\cT)\cdots\dF_{\ga_\ell}(\cT) z^{|\cT|}}
=\sumn q_n 
\E\bigsqpar{\dF_{\ga_1}(\ctn)\cdots\dF_{\ga_\ell}(\ctn)} z^n
.\end{align}
When $\ell=1$ and $2$, these are the same as $M_{\ga_1}(z)$ or $M_{1,1}(z)$
in the notation used above.
The recursion \eqref{lhz} extends as follows.
We write again
\begin{align}
    M_{\ga_1,\dots,\ga_\ell}(z)
=\frac{zy'(z)}{y(z)} R_{\ga_1,\dots,\ga_\ell}(z).
\end{align}
Then, by a straightforward extension of the proof of \cite[Lemma 12.4]{SJ359},
\cf{}\eqref{lhz},
\begin{align}\label{rec}
   R_{\ga_1,\dots,\ga_\ell}(z)
=\sum_{m=0}^\ell\frac{1}{m!}
\sum B_{\ga_{i_1}}(z)\odot\cdots\odot B_{\ga_{i_{q}}}(z)\odot
\bigsqpar{zM_{A_1}(z)\dotsm M_{A_m}(z)\Phi^{(m)}\yz}
\end{align}
where we sum over all partitions of $[\ell]:=\set{1,\dots,\ell}$ into an ordered
sequence of $m+1$ sets
$I_0,\dots,I_m$ with $I_1,\dots,I_m$ neither empty nor equal to the full set
$[\ell]$
(while $I_0$ may be empty or equal to $[\ell]$),
and $i_j$ are defined by
$I_0=\set{i_1,\dots,i_q}$ and, for $1\le j\le m$,
$A_j$ is the sequence $(\ga_i:i\in I_j)$.

As in \refL{Lth}, it follows by induction that for any sequence
$A=(\ga_1,\dots,\ga_\ell)$ of length $|A|=\ell\ge2$, 
\begin{align}
\label{lth+}
  M_{A}(z)
&=
\sum_{\gb,k}\kk_{A;\gb,k} (1-z)^{(1-\ell)/2+\gb} L(z)^k
+\Oz{\frac12(1-\ell)+\gdde}
\\\label{lthli+}
&=
\sum_{\gb,k}\hkk_{A;\gb,k}\Li_{(3-\ell)/2+\gb,k}(z)
+\Oz{\frac12(1-\ell)+\gdde}
,\end{align}
where we sum over 
$0\le k\le \ell/2$
and all $\gb$ such that $-\gb$ equals the sum of some subsequence of $A$.
Moreover, as above, it is seen by induction
that the only non-zero coefficients 
with $k=\ell/2$ have $\gb=0$, and that they appear only when $\ell$ is even
and $A$ is balanced in the sense that it
can be partitioned into $\ell/2$ pairs $\set{\ga_i,-\ga_i}$.
We now write $\kky_A:=\kk_{A;0,k}$ if $A$ is balanced with $|A|=2k$.
(We let $\kky_A:=0$ if $A$ is not balanced.)
These leading terms come from the case $m=2$ and $q=0$ in \eqref{rec}, and
we obtain the recurrence, for $|A|\ge4$,
\begin{align}\label{qkkya}
  \kky_A = 2^{-3/2}\gs \sum \kky_{A_1}\kky_{A_2},
\end{align}
summing over all partitions of $A$ into two nonempty sets $A_1$ and $A_2$
that both are balanced.

It follows by induction from \eqref{qkkya} that if $|A|=2k\ge2$, then 
$\kky_A$ can be written as a sum
\begin{align}\label{per}
  \kky_A= \bigpar{2^{-3/2}\gs}^{k-1} \sum \prod_{j=1}^k\kky_{A_j},  
\end{align}
where we sum
over full 
binary trees with $k$ leaves, where
each leaf is labelled by a  pair $I_j$ of indices 
such that $I_1,\dots,I_k$ form a partition of $[2k]$, and furthermore 
the corresponding sets $A_j$ are balanced, 
\ie, $\ga_i+\ga_{i'}=0$ if $I_j=\set{i,i'}$.

Let $A=(\ga_1,\dots,\ga_{2k})$ consist of the numbers $\ii t_j$ and $-\ii t_j$
repeated $k_j$ times each, 
for $j=1,\dots,r$,
where $t_1,\dots,t_r$ are distinct and positive;
thus $|A|=2k$ with $k=\sum_j k_j$.
Then 
there are $\prod_j k_j!$ ways to partition $A$ into balanced pairs, and
for each binary tree with $k$ leaves,
these pairs can be assigned to the $k$ leaves in $k!$ ways.
Each tree and each assignment of balanced pairs $A_i$ gives the same
contribution to the sum \eqref{per}, and we obtain,
since there are $C_{k-1}=(2k-2)!/(k!(k-1)!)$ full binary trees with $k$ leaves,
\begin{align}\label{paab0}
  \kky_A= \bigpar{2^{-3/2}\gs}^{k-1} \frac{(2k-2)!}{(k-1)!}
\prod_{j=1}^r{\left[ (\kky_{\set{\pm\ii t_j}})^{k_j}k_j! \right]}.
\end{align}

Let $\gsst$ be the variance of $\zetat$ in \eqref{tit2}.
For the case $A=\set{\ii t,-\ii t}$, \refL{Lth} applies and we have by
\eqref{ltha} and \eqref{lthkk1}, in the present  notation,
\begin{align}\label{qkky2}
\kky_{\set{\pm\ii t}}=2\qqw\gs\qw\gsst.  
\end{align}
Hence, \eqref{paab0} yields 
\begin{align}\label{paab1}
  \kky_A= 2^{-2k+\frac32}\gs^{-1} \frac{(2k-2)!}{(k-1)!}
\prod_{j=1}^r \left( \gs_{\ii t_j}^{2k_j} k_j! \right).
\end{align}
Since $\hkk_{A;0,k}=\gG(k-\frac12)\qw\kky_A$, we finally obtain from
\eqref{lthli+}, using \eqref{qn}, that
\begin{align}\label{paab2}
n^{-k}  \E\bigsqpar{\dF_{\ga_1}(\ctn)\cdots\dF_{\ga_{2k}}(\ctn)}
&\to 2^{-2k+2}\sqrt\pi \frac{(2k-2)!}{\gG(k-\frac12)(k-1)!}
\prod_{j=1}^r \left( \gs_{\ii t_j}^{2k_j} k_j! \right)
\notag\\&
=\prod_{j=1}^r \left( \gs_{\ii t_j}^{2k_j}k_j! \right),
\end{align}
which equals the corresponding mixed moment
$\E\bigpar{\zeta_{\ga_1}\cdots\zeta_{\ga_{2k}}}
=\prod_j\E|\zeta_{\ii t_j}|^{2k_j}$, 
see \eqref{cnormal}.
Similarly, all mixed moments with unbalanced indices converge after
normalization to 0.
Hence, the result follows by the method of moments.
\end{proof}

Note that the combinatorial argument in the final part of the proof 
(restricted to the case $r=1$) yields an alternative proof that the recursion
\eqref{winkk} is solved by \eqref{alla}--\eqref{id1}.
Conversely, the argument above without detailed counting of possibilities
shows that the \lhs{} of \eqref{paab2} converges to $c_k$ times the \rhs,
for some combinatorial constant $c_k$ not depending on $k_1,\dots k_r$.
Since \eqref{itussi} shows that the formula is correct for $r=1$, we must
have $c_k=1$, and thus \eqref{paab2} holds.

\section{Negative real part}\label{NEG}
\ccreset

In this section, we consider the case that $\ga$ in \eqref{Xn} has negative real part. Applying the same approach as in previous sections, we prove convergence of all moments for the normalized random variable. As before, we assume throughout the section 
that \eqref{2+gd} holds with $0<\gd<1$. Again, we set
\begin{align}
  b_n:=n^\ga-\mu(\ga),
\end{align}
with the generating function
\begin{align}\label{Bz5}
B(z)&= B_\ga(z):=\sumn b_nz^n =\Li_{-\ga}(z)-\mu(\ga)\Li_0(z).
\end{align}
In contrast to \refS{S:I}, the term $\mu(\ga)\Li_0(z) = \mu(\ga)z(1-z)^{-1}$ now dominates.
For later convenience, we let
$\eta := \min(-\rea,\,\gd/2)$, and note that $0<\eta<\frac12$ (assuming
again $\gd<1$ as we may).
Then \eqref{lia} implies
\begin{align}\label{adv2}
B(z) = -\mu(\ga) (1-z)^{-1} + \Oz{-1+\eta}.
\end{align}
This is even true for $\ga\in\set{-1,-2,\dots}$, where logarithmic terms occur in the asymptotic expansion of $\Li_{-\ga}$, due to the aforementioned fact that $\eta < \frac12$. 

Once again, we let
$\dF(T)=\dF_\ga(T)$ denote the additive functional defined by the toll function
$\df_\ga(T):=b_{|T|}$, so that
\begin{align}
 \dF_\ga(\tn)=X_n(\ga)-n\mu(\ga). 
\end{align}

\subsection{The mean}\label{NEGmean}

We use the same notation for the generating function of the mean as 
in \refS{S:I}, \ie,
\begin{align}
M_\ga(z):=\E \bigsqpar{\dF_\ga(\cT) z^{|\cT|}}
=\sumn q_n \E[\dF_\ga(\tn)] z^n,
\end{align}
and note that \eqref{sx1} still holds:
\begin{align}
  M_\ga(z) =\frac{zy'(z)}{y(z)}\cdot\bigpar{ B_\ga(z)\odot y(z)}.
\end{align}
Thus $M_\ga(z)$ is still \gda.
In analogy with \eqref{sixten2}, we now have
\begin{align}
  B_\ga(z) \odot y(z )
&=  \frac{1}{\sqrt{2\pi}\gs}\Li_{3/2-\ga}(z)
-\frac{\mu(\ga)}{\sqrt{2\pi}\gs}\Li_{3/2}(z)
+\cc+\Oz{\frac12+\frac\gd2}
\notag\\&=
2\qq\gs\qw\mu(\ga)(1-z)\qq
+\cc+\Oz{\frac12+\eta}.
\end{align}
Moreover, \eqref{xhga} still holds, so $\ccx = 0$. Combining this with \eqref{lgdyy} now yields
\begin{align}\label{meanneg}
  M_\ga(z) = \gs\qww \mu(\ga) + \Oz{\eta}.
\end{align}
Applying singularity analysis and \eqref{qn}, we find that
\begin{align} \label{EXn5}
 \E[\dF_\ga(\tn)] = \On{\frac12-\eta}
\end{align}
or equivalently
\begin{align}\label{ixq-}
  \E X_n(\ga)=\mu(\ga)n + \On{\frac12-\eta}.
\end{align}

\subsection{Higher moments}\label{NEGmom}

As in \refS{S:Imom}, we consider the mixed moments 
of $\dF_{\ga_1}(\tn)$ and $\dF_{\ga_2}(\tn)$ for two complex numbers $\ga_1$ and $\ga_2$ that are now both assumed to have negative real part. In particular, this includes the special case that $\ga_2 = \bga_1$. We are thus interested in the generating function
\begin{align}
  M_{\ell_1,\ell_2}(z)
&:=\E \bigsqpar{\dF_{\ga_1}(\cT)^{\ell_1}\dF_{\ga_2}(\cT)^{\ell_2} z^{|\cT|}}
\end{align}
for integers $\ell_1,\ell_2 \ge 0$, \cf\,\eqref{Mell2}. In particular, we have $M_{1,0}=M_{\ga_1}$ and $M_{0,1}=M_{\ga_2}$. 
Set $\eta := \min(-\rea_1,-\rea_2,\,\gd/2)$ (again noting that $\eta < \frac12$). Then by
\eqref{meanneg}
we have
\begin{align}\label{m1001}
 M_{1,0}(z) = \gs\qww \mu(\ga_1) + \Oz{\eta} \text{ and }  M_{0,1}(z) = \gs\qww \mu(\ga_2) + \Oz{\eta}.
\end{align}
In order to deal with higher moments, we make use of the recursion \eqref{lhz}. Let us start with 
second-order moments:\ here, we obtain
\begin{align}\label{m2as1}
  M_{1,1}(z) &= \frac{z y'(z)}{y(z)} \left[ B_{\ga_1}(z)\odot B_{\ga_2}(z)\odot y(z)
+ B_{\ga_1}(z)\odot (zM_{0,1}(z)\Phi'(y(z))) \right.
\notag\\&\qquad
\left.+ B_{\ga_2}(z)\odot (zM_{1,0}(z)\Phi'(y(z)))
+ zM_{1,0}(z)M_{0,1}(z)\Phi''(y(z)) \right].
\end{align}
In view of \eqref{meanneg}, \eqref{lgdy}, \eqref{lz0a}, and \eqref{lz0b}, the functions $y(z)$, $zM_{0,1}(z)\Phi'\bigpar{y(z)}$, $zM_{1,0}(z)\Phi'\bigpar{y(z)}$, and $z M_{1,0}(z)M_{0,1}(z) \Phi''\bigpar{y(z)}$ are all of the form $c + \Oz{\eta}$, and taking the Hadamard product with $B_{\ga_1}(z)$ or $B_{\ga_2}(z)$ does not change this property. Combining this with \eqref{lgdyy} we conclude that there is a constant $\kk_{1,1}$ such that
\begin{align}\label{m2as1x}
M_{1,1}(z) = 2\qqw \gs\qw \kk_{1,1} (1-z)^{-1/2} + \Oz{-\frac12+\eta},
\end{align}
which implies
 by virtue of singularity analysis and~\eqref{qn} that
\begin{align}
\E[\dF_{\ga_1}(\tn)\dF_{\ga_2}(\tn)] = \kk_{1,1}\,n + \On{1 - \eta}.
\end{align}
We can obtain the functions $M_{2,0}(z)$ and $M_{0,2}(z)$ as special cases of $M_{1,1}(z)$ where $\ga_1 = \ga_2$. Hence there are also constants $\kk_{2,0}$ and $\kk_{0,2}$ such that
\begin{align}\label{m2as2}
M_{2,0}(z) = 2\qqw \gs\qw \kk_{2,0} (1-z)^{-1/2} + \Oz{-\frac12+\eta}
\end{align}
and
\begin{align}\label{m2as3}
M_{0,2}(z) = 2\qqw \gs\qw \kk_{0,2} (1-z)^{-1/2} + \Oz{-\frac12+\eta},
\end{align}
and thus
\begin{align}
\E[\dF_{\ga_1}(\tn)^2] = \kk_{2,0}\,n + \On{1 - \eta} \text{ and } \E[\dF_{\ga_2}(\tn)^2] = \kk_{0,2}\,n + \On{1 - \eta}.
\end{align}
We will use these as the base case of an inductive proof of the following lemma.

\begin{lemma}\label{mellsa}
Suppose that $\rea_1 < 0$ and $\rea_2 < 0$, and let 
\begin{align}
  \eta = \min(-\rea_1,-\rea_2,\,\gd/2)
\end{align} 
be as above. 
Then, for all non-negative integers $\ell$ and $r$ with $s = \ell+r \ge 1$, the function $M_{\ell,r}(z)$ is \gda{} and we have
\begin{align} \label{Mz5}
M_{\ell,r}(z) = \hkk_{\ell,r} (1-z)^{(1-s)/2} + \Oz{(1-s)/2+\eta},
\end{align}
where $\hkk_{1,0}=\gs\qww\mu(\ga_1)$, $\hkk_{0,1}=\gs\qww\mu(\ga_2)$, and, for $s\ge2$,
\begin{align}\label{kappalr}
\hkk_{\ell,r} = \frac{(s-3)!!}{\gs 2^{(s-1)/2}} \sum_{\substack{j=0 \\ j \equiv \ell \bmod 2}}^{\ell \land r}
\binom{\ell}{j} \binom{r}{j} j!\, (\ell-j-1)!!\,(r-j-1)!!\,\kk_{1,1}^{j} \kk_{2,0}^{(\ell-j)/2}\kk_{0,2}^{(r-j)/2}
\end{align}
if $s$ is even, and $\hkk_{\ell,r} = 0$ otherwise.
\end{lemma}

\begin{proof}
We 
prove the statement by induction on $s = \ell + r$. Note that \eqref{m1001} as well as \eqref{m2as1x}, \eqref{m2as2}, and \eqref{m2as3} are precisely the cases $s=1$ and $s = 2$, respectively. 

For the induction step, we take
$s \ge 3$ and use recursion \eqref{lhz}. It follows immediately from this recursion that all $M_{\ell,r}$ are $\Delta$-analytic, so we focus on the asymptotic behavior at $1$. Let us first consider the product
\begin{align}
zM_{\ell_1,r_1}(z)\dotsm M_{\ell_m,r_m}(z)\Phi\xxm\yz,
\end{align}
where all $\ell_i$ and $r_i$ are non-negative integers, $1 \le \ell_i + r_i < s$ for 
every $i \ge 1$, $\ell_0 + \ell_1 + \cdots + \ell_m = \ell$, and $r_0 + r_1 + \cdots + r_m = r$. By the induction hypothesis, $M_{\ell_i,r_i}(z) = \Oz{(1-\ell_i-r_i)/2}$ for all $i \ge 1$, which can be improved to $M_{\ell_i,r_i}(z) = \Oz{(1-\ell_i-r_i)/2 + \eta}$ if $\ell_i + r_i$ is odd and greater than $1$. Combining with \eqref{lz0c}, we obtain
\begin{align}
zM_{\ell_1,r_1}(z)\dotsm M_{\ell_m,r_m}(z)\Phi\xxm\yz &= \Oz{(m-\ell_1 - \cdots - \ell_m-r_1-\cdots-r_m)/2 + \gdd + 1 - m/2}\notag  \\ 
&= \Oz{(\ell_0 + r_0 - \ell - r)/2 + 1 + \eta}
\end{align}
for $m \ge 3$. This estimate continues to hold after taking the Hadamard
product with $B_{\ga_1}(z)^{\odot\ell_0} \odot B_{\ga_2}(z)^{\odot r_0}$,
and the factor $\frac{z y'(z)}{y(z)}$ in \eqref{lhz} 
contributes $-\frac12$ to the exponent by \eqref{lgdyy}. Since
$\ell_0$ and $r_0$ are non-negative, it follows that the total contribution of all
terms with $m \ge 3$ is $\Oz{(1- s)/2 + \eta}$ and thus negligible. We
can therefore focus on the cases $m = 0$, $m=1$, and $m=2$. Here,
$\Phi\xxm\yz$ is 
$O(1)$
in all cases by 
\eqref{lz00}--\eqref{lz0b}, and we obtain 
\begin{align}
zM_{\ell_1,r_1}(z)\dotsm M_{\ell_m,r_m}(z)\Phi\xxm\yz &= \Oz{(m-\ell_1 - \cdots - \ell_m-r_1-\cdots-r_m)/2} \notag\\
&= \Oz{(m + \ell_0 + r_0 - \ell - r)/2}.
\end{align}
Terms with
 $m + \ell_0 + r_0 \ge 3$ are negligible for the same reason as before. Likewise, terms with $m + \ell_0 + r_0 = 2$ are negligible if at least one of the sums $\ell_i + r_i$ with $i \ge 1$ is odd and greater than $1$, as we can then improve the bound on $M_{\ell_i,r_i}(z)$. Let us determine all remaining possibilities:
\begin{itemize}
\item $m = 0$ implies $m + \ell_0 + r_0 = \ell + r = s \ge 3$, so we have already accounted for this negligible case.
\item $m = 1$ gives us $\ell_0 + \ell_1 = \ell$ and $r_0 + r_1 = r$ with $1 \le \ell_1 + r_1 < \ell + r$, thus $\ell_0 + r_0 \ge 1$. So we have $(\ell_0,\ell_1,r_0,r_1) = (1,\ell-1,0,r)$ and $(\ell_0,\ell_1,r_0,r_1) = (0,\ell,1,r-1)$ as the only two relevant possibilities in this case.
\item Finally, if $m = 2$, we must have $\ell_0 = r_0 = 0$ and $\ell_1 + \ell_2 = \ell$ and $r_1 + r_2 = r$.
\end{itemize}
Now we divide the argument into two subcases, according as $s = \ell + r$ is even or odd.

\emph{Odd $s \ge 3$.} 
If $m = 2$, $\ell_0 = r_0 = 0$, and $\ell_1 + \ell_2 + r_1  + r_2 = \ell + r  = s$, then either $\ell_1 + r_1$ or $\ell_2 + r_2$ is odd.
Thus the corresponding term is asymptotically negligible unless $\ell_1 + r_1 = 1$ or $\ell_2 + r_2 = 1$.  So in this case, there are only four terms
that might be asymptotically 
relevant: 
\begin{align}
  (\ell_1,\ell_2,r_1,r_2) \in \{(1,\ell-1,0,r), (\ell-1,1,r,0), (0,\ell,1,r-1), (\ell,0,r-1,1)\}.
\end{align}
In addition, $m = 1$ contributes with two terms as mentioned above. Thus we obtain
\begin{align}\label{msodd}
M_{\ell,r}(z) &= \frac{z y'(z)}{y(z)} \left[ \ell B_{\ga_1}(z) \odot \left(zM_{\ell-1,r}(z) \Phi'\bigpar{y(z)}\right) + r B_{\ga_2}(z) \odot \left(zM_{\ell,r-1}(z) \Phi'\bigpar{y(z)}\right) \right. \notag \\
&\qquad \left. + \ell z M_{1,0}(z) M_{\ell-1,r}(z) \Phi''\bigpar{y(z)} + r z M_{0,1}(z) M_{\ell,r-1}(z) \Phi''\bigpar{y(z)} \right] \notag \\
&\quad +  \Oz{(1-s)/2+\eta}.
\end{align}
By the induction hypothesis, $M_{\ell-1,r}(z) =  \hkk_{\ell-1,r} (1-z)^{1-\frac{s}{2}} + \Oz{1-\frac{s}{2}+\eta}$. 
Consequently, using
\eqref{meanneg}, \eqref{lz0a}, and \eqref{lz0b}, we get
\begin{align}
zM_{\ell-1,r}(z) \Phi'\bigpar{y(z)} &= \hkk_{\ell-1,r} (1-z)^{1-\frac{s}{2}} + \Oz{1-\frac{s}{2}+\eta}, \\
z M_{1,0}(z) M_{\ell-1,r}(z) \Phi''\bigpar{y(z)} &= \mu(\ga_1) \hkk_{\ell-1,r} (1-z)^{1-\frac{s}{2}} + \Oz{1-\frac{s}{2}+\eta}.
\end{align}
Recall from~\eqref{Bz5} that $B_{\ga_1}(z) = \Li_{-\ga_1}(z)-\mu(\ga_1)\Li_0(z)$. 
Applying the Hadamard product gives us, using \eqref{lia}, \eqref{lili}, and \refL{LFFK}, 
\begin{align}
B_{\ga_1}(z) \odot \left(zM_{\ell-1,r}(z) \Phi'\bigpar{y(z)}\right) = -\mu(\ga_1) \hkk_{\ell-1,r} (1-z)^{1-\frac{s}{2}} 
+ \Oz{1-\frac{s}{2}+\eta},
\end{align}
so the first and third terms in \eqref{msodd} effectively cancel, and the same argument applies to the second and fourth terms. Hence we have proven the desired statement in the case that $s$ is odd.

\emph{Even $s \ge 4$.}
In this case, we can neglect the terms with $m = 1$ and $\ell_1 + r_1 = \ell + r -1 = s-1$, since $s-1$ is odd and greater than $1$. Thus only terms with $m = 2$ and $\ell_0 = r_0 = 0$ matter. For the same reason, we can ignore all terms where $\ell_1+r_1$ and $\ell_2+r_2$ are odd: at least one of them has to be greater then $1$, making all such terms asymptotically negligible.
Hence we obtain
\begin{align}
M_{\ell,r}(z) &= \frac{z y'(z)}{y(z)} \cdot \frac12 \sum_{\substack{\ell_1,\ell_2,r_1,r_2 \\ \ell_1+\ell_2 = \ell,\,r_1+r_2 = r \\ \ell_i + r_i \text{ even and }> 0}} \binom{\ell}{\ell_1} \binom{r}{r_1} z M_{\ell_1,r_1}(z) M_{\ell_2,r_2}(z) \Phi''\bigpar{y(z)} \notag \\
&\qquad + \Oz{(1-s)/2+\eta}. \label{mseven}
\end{align}
Let us write $\sumo$ for the sum in~\eqref{mseven}. Plugging in \eqref{lgdyy}, \eqref{lz0b}, and the induction hypothesis, we obtain
\begin{align}
M_{\ell,r}(z) = 2^{-3/2} \gs  \sumo \binom{\ell}{\ell_1} \binom{r}{r_1} \hkk_{\ell_1,r_1}\hkk_{\ell_2,r_2} (1-z)^{(1-s)/2} + \Oz{(1-s)/2+\eta}.
\end{align}
Thus we have completed the induction for~\eqref{Mz5} with
\begin{align}\label{kapparec}
\hkk_{\ell,r} = 2^{-3/2} \gs \sumo \binom{\ell}{\ell_1} \binom{r}{r_1} \hkk_{\ell_1,r_1}\hkk_{\ell_2,r_2}.
\end{align}

In order to verify the formula~\eqref{kappalr} for $\hkk_{\ell,r}$ given in the statement of the lemma, 
in light of~\eqref{m2as1x}, \eqref{m2as2}, and~\eqref{m2as3}
we need only show that
$\hkk_{\ell,r}$ as defined in \eqref{kappalr} satisfies the recursion \eqref{kapparec}. This is easy to achieve by means of generating functions, as follows.  Set
\begin{align}
K(x,y) &:= \sum_{\substack{s \ge 2 \\ s \text{ even}}} \sum_{\ell+r = s} \hkk_{\ell,r} \frac{x^{\ell}}{\ell!} \frac{y^r}{r!} \notag \\
&= \sum_{\substack{s \ge 2 \\ s \text{ even}}} \sum_{\ell+r = s} \frac{(s-3)!!}{\gs 2^{(s-1)/2}}
\sum_{\substack{j=0 \\ j \equiv \ell \bmod 2}}^{\ell \land r}
\binom{\ell}{j} \binom{r}{j} j!\, (\ell-j-1)!!\,(r-j-1)!! \notag \\
&\qquad \qquad \qquad \cdot\kk_{1,1}^{j} \kk_{2,0}^{(\ell-j)/2}\kk_{0,2}^{(r-j)/2} \frac{x^{\ell}}{\ell!} \frac{y^r}{r!}.
\end{align}
Setting $\ell-j = 2a$ and $r-j = 2b$, this can be rewritten 
as
\begin{align}
K(x,y) &= \sum_{\substack{s \ge 2 \\ s \text{ even}}} \frac{(s-3)!!}{\gs 2^{(s-1)/2}} \sum_{\substack{a,b,j \ge 0: \\ 
a+b+j = s/2}} \frac{\kk_{1,1}^{j} \kk_{2,0}^{a}\kk_{0,2}^{b} x^{j+2a} y^{j+2b}}{j!\,a!\,b!\,2^{a+b}} \notag \\
&= \sum_{\substack{s \ge 2 \\ s \text{ even}}} \frac{(s-3)!!}{\gs 2^{(s-1)/2}(s/2)!} \left( \frac{\kk_{2,0}\,x^2}{2} + \kk_{1,1}\,xy + \frac{\kk_{0,2}\,y^2}{2} \right)^{s/2} \notag \\
&= \frac{\sqrt{2}}{\gs} \sum_{t \ge 1} \frac{(2t-3)!!}{t!\, 2^{2t}} \left( \kk_{2,0}\,x^2 + 2\kk_{1,1}\,xy + \kk_{0,2}\,y^2 \right)^t \notag \\
&= \frac{\sqrt{2}}{\gs} - \frac{1}{\gs} \sqrt{2 - \left( \kk_{2,0}\,x^2 + 2\kk_{1,1}\,xy + \kk_{0,2}\,y^2 \right)}.
\end{align}
The recursion \eqref{kapparec} now follows by comparing coefficients of $x^{\ell}y^r$ in the identity
\begin{align}
K(x,y) = 2^{-3/2} \gs K(x,y)^2 + \frac{\kk_{2,0}\,x^2 + 2\kk_{1,1}\,xy + \kk_{0,2}\,y^2}{2^{3/2} \gs}.
\end{align}
This completes the proof of the lemma.
\end{proof}

So the functions $M_{\ell,r}(z)$ are amenable to singularity analysis, and we obtain the following theorem as an immediate application.

\begin{theorem}\label{TNEGmom}
Suppose that  $\rea_1 < 0$ and $\rea_2 < 0$. Then there exist constants $\kk_{2,0}$, $\kk_{1,1}$, and $\kk_{0,2}$ such that, for all non-negative integers $\ell$ and $r$,
\begin{align}\label{tnegmom}
&\frac{\E[\dF_{\ga_1}(\tn)^{\ell}\dF_{\ga_2}(\tn)^{r}]}{n^{(\ell+r)/2}} \notag \\
&\qquad \to  \sum_{\substack{j=0 \\ j \equiv \ell \bmod 2}}^{\ell \land r}
\binom{\ell}{j} \binom{r}{j} j!\, (\ell-j-1)!!\,(r-j-1)!!\, \kk_{1,1}^{j} \kk_{2,0}^{(\ell-j)/2}\kk_{0,2}^{(r-j)/2} 
\end{align}
as $n \to \infty$ if $\ell+r$ is even, and $\frac{\E[\dF_{\ga_1}(\tn)^{\ell}\dF_{\ga_2}(\tn)^{r}]}{n^{(\ell+r)/2}} \to 0$ otherwise.
\end{theorem}

\begin{proof}
In view of Lemma~\ref{mellsa}, singularity analysis gives us
\begin{align}
[z^n] M_{\ell,r}(z) = \frac{\hkk_{\ell,r}}{\gG((s-1)/2)} n^{(s-3)/2} + \On{(s-3)/2-\eta}
\end{align}
for $s = \ell + r \ge 2$, so, using~\eqref{qn},
\begin{align}
\E[\dF_{\ga_1}(\tn)^{\ell}\dF_{\ga_2}(\tn)^{r}] = \frac{[z^n] M_{\ell,r}(z)}{q_n} = \frac{\sqrt{2\pi}\gs\hkk_{\ell,r}}{\gG((s-1)/2)} n^{s/2} + \On{s/2-\eta}.
\end{align}
Since $\gG((s-1)/2) =2^{1-(s/2)}\sqrt{\pi}(s-3)!!$ for even $s$ (recall~\eqref{semi}), the statement follows immediately from the formula for $\hkk_{\ell,r}$ in Lemma~\ref{mellsa} for all $s \ge 2$ and from~\eqref{EXn5} for $s = 1$.
\end{proof}

The following lemma will be used in the proof of \refT{TNEG} to establish that the limiting variance is positive.  Recall the notation \eqref{Xn} and $q_k = \P(|\cT| = k)$.

\begin{lemma}\label{Lnondeterm}
Consider any complex~$\ga$ with $\rea \neq 0$.  Then there exists $k$ such that $q_k > 0$ and $F_{\ga}(\cT_k)$ is not deterministic.
\end{lemma}

\begin{proof}
We know that $p_0 > 0$ and that $p_j > 0$ for some $j \geq 2$.  Fix such a value~$j$.  Let $k = 3 j + 1 \geq 7$.  Consider 
two realizations of the random tree $\cT_k$, each of which has positive probability.  Tree 1 has $j$ children of the root, and precisely two of those~$j$ children have $j$ children each; the other $j - 2$ have no children.  Tree 2 also has $j$ children of the root; precisely one of those $j$ children (call it child 1) has $j$ children, while the other $j - 1$ have no children; precisely one of the children of child 1 has $j$ children, while the others have no children.

Then the values of $F_{\ga}$ for Tree 1 and Tree 2 are, respectively,
\begin{align}
3 j - 2 + 2 (j + 1)^{\ga} + (3 j + 1)^{\ga}
\end{align}
and
\begin{align}
3 j  - 2 + (j + 1)^{\ga} + (2 j + 1)^{\ga} + (3 j + 1)^{\ga}.
\end{align}
These values can't be equal, because otherwise we would have $(j + 1)^{\ga} = (2 j + 1)^{\ga}$; but the two numbers here have unequal absolute values.  
\end{proof}

\begin{proof}[Proof of \refT{TNEG}]
The limit in \eqref{tnegmom} equals the mixed moment 
$\E \bigsqpar{\zeta_1^\ell\zeta_2^r}$, where $\zeta_1$ and $\zeta_2$ have a
joint complex normal distribution and 
$\E\zeta_1^2=\kk_{2,0}$,
$\E\zeta_1\zeta_2=\kk_{1,1}$,    
and $\E\zeta_2^2=\kk_{0,2}$; this follows by Wick's theorem
\cite[Theorem 1.28 or Theorem 1.36]{SJIII} by noting that
the factor
$\binom{\ell}{j} \binom{r}{j} j!\, (\ell-j-1)!!\,(r-j-1)!!$
in \eqref{tnegmom}
is the number of perfect matchings of $\ell$ (labelled) copies of $\zeta_1$ and
$r$ copies of $\zeta_2$ such that there are $j$ pairs $(\zeta_1,\zeta_2)$.

Hence, \refT{TNEG}, except for the assertion of positive variance addressed next, follows by the method of moments, taking $\ga_1:=\ga$ and $\ga_2:=\bga$, \cf{} \refR{Rmom}.

We already know from \refT{TNEGmom} that $\Var F_{\ga}(\tn) = a\,n + o(n)$ for some $a \geq 0$; we need only show that $a > 0$.  Fix $k$ as in \refL{Lnondeterm}.  Write $v_k > 0$ for the variance of $F_{\ga}(\cT_k)$.  Let $N_{n, k}$ denote the number of fringe subtrees of size $k$ in $\tn$.  It  follows from \cite[Theorem 1.5(i)]{SJ285}
that
\begin{align}
\E N_{n, k} \sim q_k n
\end{align}
as $n \to \infty$.  If for $\tn$ we condition on $N_{n, k} = m$ and all of $\tn$ except for fringe subtrees of size $k$, then the conditional variance of $F_{\ga}(\tn)$ is the variance of the sum of $m$ independent copies of $F_{\ga}(\cT_k)$, namely, $m v_k$.  Thus
\begin{align}
\Var F_{\ga}(\tn) \geq v_k \E N_{n, k} \geq (1 + o(1)) v_k q_k n,
\end{align}
so the constant $a$ mentioned at the start of this paragraph satisfies $a \geq v_k q_k > 0$.
\end{proof}

\begin{remark}
We recall that asymptotic normality of $X_n(\ga)$, 
or equivalently of $\dF_{\ga}(\tn)$, is already proven in \cite[Theorem~1.1]{SJ359}.
Furthermore, \cite[Section~5]{SJ359} shows joint asymptotic normality for several
$\ga$ with $\rea<0$, 
which for the case of two values $\ga_1$ and $\ga_2$ is consistent
with \eqref{tnegmom} (by the argument in the proof of \refT{TNEG} above).
It would certainly be possible to generalize the moment convergence results
in this section to convergence of mixed moments for  combinations of
several $\ga_i$, similarly to \refS{S:Ijoint},
including also the possibility $\Re\ga_i\ge0$ for some values of~$i$.
However, this would require a lengthy case distinction (depending on the
signs of the values $\Re \ga_i$), so we did not perform these calculations
explicitly.
Instead we just note that if we consider only the case
$\Re\ga_i<0$, then 
convergence of all mixed moments follows from
the joint convergence in \eqref{t1.1} for several $\ga_i$
shown in \cite[Section~5]{SJ359} 
together with the uniform integrability of 
$|n\qqw [X_n(\ga) - \mu(\ga)n]|^r$ for 
arbitrary $r>0$ that follows from \refT{TNEG} (see \refR{Rmom}).
\end{remark}

\section{Fractional moments (mainly of negative order) of tree-size:\\ comparisons across offspring distributions}
\label{S:comp}

Recall from \cite[Theorem~1.7]{SJ359} that the $\ga$th moment $\mu(\ga) = \E |\cT|^{\ga}$ of tree size defined at~\eqref{mu} is the slope in the lead-order linear approximation $\mu(\ga) n$ of $\E X_n(\ga)$ whenever $\Re \ga < \half$; and from \refT{TNEG} that this linear approximation suffices as a centering for $X_n(\ga)$ in order to obtain a normal limit distribution when $\Re(\ga) < 0$.  (See also \refR{Rcenters}.)  It is therefore of interest to compute 
$\mu(\ga)$ and, similarly, the constant $\mu' = \E \log |\cT|$ defined at~\eqref{xd}, which serves as the centering slope in \refT{TZ}. 

In \cite[Appendix~A]{SJ359} it is noted that although $\mu(\ga)$ can be evaluated numerically, no exact values for important examples of \GW\ trees are known in any simple form except in the case that $\ga$ is a negative integer.  This section is motivated by our having noticed that for all such values (for small~$k$) reported for four examples in that appendix, $\mu(-k)$ is smallest for binary trees \cite[Example~A.3]{SJ359}, second smallest for labelled trees \cite[Example~A.1]{SJ359}, second largest for full binary trees \cite[Example~A.4]{SJ359}, and largest for ordered trees \cite[Example~A.2]{SJ359}.  We wanted to understand why this ordering occurs and whether any such ordering could be predicted for the values $\mu'$ defined at~\eqref{xd}.

In \refS{S:comp_theo} we give a sufficient condition (\eqref{strict_Phi_order} in \refT{Tstrict_order}) for such (strict) orderings that is fairly easy to check.  In \refS{S:ex} we give a class of examples extending the four in \cite[Appendix~A]{SJ359} where this condition is met.  In \refS{S:num} we discuss numerical computation of $\mu'$, which we carry out for the four examples in \cite[Appendix~A]{SJ359} and some additional examples.

The results of this \refS{S:comp} do not require~\eqref{2+gd}.    

\subsection{Comparison theory}\label{S:comp_theo}

The main results of this section are in 
\refT{Tstrict_order}.
Working toward those results, we begin by recalling from \cite[(A.6)]{SJ359} (where~$y$ is called ``$g$'' and~\eqref{2+gd} is not required) that for $\Re \ga < \half$ we have
\begin{equation}\label{A.6}
\mu(\ga) = \frac{1}{\gG(1 - \ga)} \int_0^1\! (\log \tfrac1t)^{-\ga} y'(t) \dd t.
\end{equation}

To utilize \eqref{A.6} directly, even merely to obtain inequalities across models for real~$\ga$, one needs to compute the derivative of the tree-size \pgf\ $y$, or at least to compare the functions $y'$ for the compared models.  This is nontrivial, since explicit computation of~$y'$ (or~$y$) is difficult or even infeasible in examples such as $m$-ary trees and full $m$-ary trees when $m > 2$.  Fortunately, according to~\eqref{A.6recast} (and similarly~\eqref{mu'intrecast} in regard to $\mu'$) to follow, one need only treat the simpler offspring {\pgf}(s)~$\Phi$. 

Before proceeding to our main results, we present 
a simple lemma, a recasting of~\eqref{A.6}, and a definition.

\begin{lemma}\label{Lstrict}
The function $t \mapsto t / \Phi(t)$ 
is the inverse function of $y:\oi\to\oi$, and it
increases strictly from~$0$ to~$1$ for $t \in [0, 1]$.
\end{lemma}

\begin{proof}
It is obvious from \eqref{yz} that $y(z)$ is 
continuous and strictly increasing for $z\in\oi$ with $y(0)=0$ and $y(1)=1$.
Hence its inverse is also strictly increasing from 0 to 1 on $\oi$.
Finally, \eqref{Phiy} shows that the inverse is $t \mapsto t / \Phi(t)$.
\end{proof}

We 
will henceforth write
\begin{equation}\label{rdef}
R(\eta) := \frac{1}{y^{-1}(\eta)} = \frac{\Phi(\eta)}{\eta} \in [1, \infty), \quad \eta \in (0, 1];
\end{equation}
this strictly decreasing function~$R$ will appear on several occasions in the sequel, especially in \refApp{Aconverse}.

It follows from \eqref{A.6}, \refL{Lstrict}, a change of variables from~$t$ to $\eta = y(t)$, and~\eqref{Phiy} that
\begin{equation}\label{A.6recast}
\mu(\ga) = \frac{1}{\gG(1 - \ga)} \int_0^1 [\log R(\eta)]^{-\ga} \dd \eta.
\end{equation}
Further, differentiation with respect to~$\ga$ at $\ga = 0$ gives
\begin{equation}\label{mu'intrecast}
\mu' = - \gam - \int_0^1[\log \log R(\eta)] \dd \eta.
\end{equation}

For the remainder of \refS{S:comp} we focus on real~$\ga$ and utilize the following notation.

\begin{definition}\label{Dineq}
For two real-valued functions $g_1$ and $g_2$ defined on $(0, 1)$, write $g_1 \leq g_2$ to mean that $g_1(t) \leq g_2(t)$ for all $t \in (0, 1)$; write $g_1 < g_2$ to mean that $g_1 \leq g_2$ but $g_2 \not\leq g_1$ (equivalently, that $g_1(t) \leq g_2(t)$ for all $t \in (0, 1)$, with strict inequality for 
at least one value of~$t$); and write $g_1 \prec g_2$ to mean that $g_1(t) < g_2(t)$ for all $t \in (0, 1)$.
\end{definition}

Consider two \GW{} trees, $\cT^{(1)}$ 
and $\cT^{(2)}$, with respective offspring distributions $\xi_1$ and $\xi_2$.  Denote the trees' respective 
$\Phi$-functions by $\Phi_1$ and $\Phi_2$, and use similarly subscripted notation 
for other functions associated with the trees.

We note in passing that, as a simple consequence of~\refL{Lstrict} 
whose proof is left to the reader, 
\begin{equation}\label{Phiiffy}
\Phi_1 \leq \Phi_2 \quad \mbox{if and only if} \quad y_1 \leq y_2,
\end{equation}
and hence also
\begin{equation}
\Phi_1 < \Phi_2 \quad \mbox{if and only if} \quad y_1 < y_2.
\end{equation}
The result~\eqref{Phiiffy} is perhaps of some independent interest but is used in the sequel mainly in the proof of \refT{TLaplace}.

\begin{theorem}\label{Torder}
Consider two \GW{} trees, $\cT^{(1)}$ and $\cT^{(2)}$.  Suppose
\begin{equation}
\Phi_1 \leq \Phi_2. \label{Phi_order}
\end{equation}
  \begin{thmenumerate}
  \item \label{T_neg_a}
If\/ $\ga < 0$, then
\begin{equation}\label{neg_a}
\mu_1(\ga) \leq \mu_2(\ga).
\end{equation}
  \item \label{T_pos_a}
If\/ $0 < \ga < \frac12$, then
\begin{equation}\label{pos_a}
\mu_1(\ga) \geq \mu_2(\ga).
\end{equation}
\item \label{T_shape}
The centering constants for the corresponding shape functionals satisfy
\begin{align}\label{t_shape}
  \mu_1' \geq \mu_2'.
\end{align}
  \end{thmenumerate}
\end{theorem}

\begin{proof}
This is immediate from~\eqref{A.6recast} and \eqref{mu'intrecast}.
\end{proof}

Note 
that, by considering difference quotients, each of~\ref{T_neg_a} and~\ref{T_pos_a} in \refT{Torder} implies~\ref{T_shape} there; one doesn't need the stronger hypothesis~\eqref{Phi_order} for this conclusion.

\begin{remark}
The conclusions
in \refT{Torder} do not always extend from $\mu(\ga)$ to 
$\E X_n(\ga)$ for finite~$n$.  A counterexample with $n = 3$ is provided by
taking $\xi_1 \sim 2 \Bi(1, \tfrac12)$ (corresponding to uniform full binary
trees, with $X_3(\ga)$ concentrated at $2 + 3^{\ga}$) and $\xi_2 \sim\Ge(\half)$ 
(corresponding to ordered trees, with $\E X_3(\ga) = \frac32 +
\half 2^{\ga} + 3^{\ga}$).  As shown in \refL{LPhiordt}, we have $\Phi_1
\leq \Phi_2$, but \refT{Torder}\ref{T_neg_a}--\ref{T_pos_a} 
with $\E X_3(\ga)$ in place of of $\mu(\ga)$ fails for
\emph{every} value of~$\ga$, as does \eqref{t_shape}.
\end{remark}

The converse to \refT{Torder} fails.  That is, \refT{Torder}\ref{T_neg_a}--\ref{T_pos_a} do not imply that~\eqref{Phi_order} does, too.  A counterexample is provided in \refApp{Aconverse}.
However, as the next theorem shows, \eqref{Phi_order} has for $\ga < 0$ a stronger consequence than~\refT{Torder}\ref{T_neg_a}, and this stronger consequence yields a converse result:

\begin{theorem}\label{TLaplace}
We have
\begin{equation}\label{iff_ineq}
\mbox{\rm $\E (|\cT_1| - 1 + t)^{\ga} \leq \E (|\cT_2| - 1 + t)^{\ga}$ for all integers $\ga < 0$ and all $t \in (0, \infty)$}
\end{equation} 
if and only if~\eqref{Phi_order} holds, in which case the inequality in~\eqref{iff_ineq} also holds for all real $\ga < 0$ and all $t \in (0, \infty)$. 
\end{theorem}

\begin{proof}
Setting 
$\ga = -k$ in~\eqref{A.6recast} and summing over positive integers~$k$, for complex~$z$ in the open unit disk define the function~$H$ as at \cite[(A.7)]{SJ359}:
\begin{equation}
H(z) := \E \Bigpar{1 - \frac{z}{|\cT|}}^{-1} 
= \sum_{k = 0}^{\infty} \mu(-k) z^k
= \int_0^1\!\exp\left[ z \log R(\eta) \right] \dd \eta.
\end{equation}
Changing variables (back) from~$\eta$ to 
$t = y\qw(\eta) = 1 / R(\eta)$, we then find
\begin{equation}
H(z) = \int_0^1\!t^{-z} y'(t) \dd t = 1 + z \int_0^1\!t^{- z - 1} y(t) \dd t,
\end{equation}
with the last equality, resulting from integration by parts, as noted at \cite[(A.9)]{SJ359}; thus
\begin{equation}\label{Et-z-1}
\E (|\cT| - z)^{-1} = z^{-1} (H(z) - 1) = \int_0^1\!t^{- z - 1} y(t) \dd t.
\end{equation}
Since both the first and third expressions in~\eqref{Et-z-1} are analytic for all~$z$ with $\Re z < 1$, they are equal in this halfplane.  Changing variables, we then find, for $\Re z > -1$, that
\begin{equation}\label{Et-z-1b}
\E (|\cT| + z)^{-1} = \int_0^{\infty}\!e^{- z x} y(e^{-x}) \dd x.
\end{equation}
In particular, if~\eqref{Phi_order} holds, then (recalling~\eqref{Phiiffy})
\begin{align}
\E (|\cT_1| - 1 + t)^{-1} 
&= \int_0^{\infty}\!e^{- t x} e^x y_1(e^{-x}) \dd x \notag\\ 
&\leq \int_0^{\infty}\!e^{- t x} e^x y_2(e^{-x}) \dd x
= \E (|\cT_2| - 1 + t)^{-1}
\end{align}
for real $t > 0$.

But more is true.  Let $\gD(z) := e^z [y_2(e^{-z}) - y_1(e^{-z})]$.  Then for $t > 0$ we have that
\begin{equation}\label{hrep}
h(t) := \E (|\cT_2| - 1 + t)^{-1} - \E (|\cT_1| - 1 + t)^{-1} = \int_0^{\infty}\!e^{- t x} \gD(x) \dd x 
\end{equation}
is the Laplace transform of the bounded continuous function~$\gD$ on $(0, \infty)$.  It 
follows from the Bernstein--Widder theorem (\eg, \cite[Theorem XIII.4.1a]{FellerII}) that~$h$ satisfies the (weak) complete monotonicity inequalities~\eqref{iff_ineq}, \ie,
\begin{equation}\label{complete_weak}
\mbox{\rm $(-1)^r h^{(r)}(t) \geq 0$ for all integers $r \geq 0$ and all $t \in (0, \infty)$},
\end{equation}
if and only if $\gD(x) \geq 0$ for a.e.\ $x > 0$, which in turn is true if and only if $y_1 \leq y_2$, or (by~\eqref{Phiiffy}) equivalently~\eqref{Phi_order}, holds.

Next, if~\eqref{Phi_order} 
holds, then for real $\ga < 0$ and $t \in (0, 1]$ we have
\begin{align}
\E (|\cT_1| - 1 + t)^{\ga} 
&= \sum_{k = 0}^{\infty} \binom{|\ga| + k - 1}{k}\,\mu_1(\ga - k)\,(1 - t)^k 
\notag\\
&\leq \sum_{k = 0}^{\infty} \binom{|\ga| + k - 1}{k}\,\mu_2(\ga - k)\,(1 - t)^k 
\notag\\
&= \E (|\cT_2| - 1 + t)^{\ga}, 
\end{align}
where the inequality holds by \refT{Torder}\ref{T_neg_a}.

Finally, if~\eqref{Phi_order} holds, then for real $\ga < 0$ and $t > 1$, \refT{Tt>1} in \refApp{Amono} implies that for 
$j \in \{1, 2\}$ we have
\begin{equation}\label{t>1}
\E (|\cT_j| - 1 + t)^{\ga} = (t - 1)^{\ga} \int_0^1\!\left[ 1 - \frac{c_j(\eta)}{\Gamma(-\ga)} \right] \dd \eta, 
\end{equation}
where $c_j$ is the incomplete gamma function value
\begin{equation}
c_j(\eta) = \int_{(t - 1) \log R_j(\eta)}^{\infty} w^{ - \ga - 1} e^{-w} \dd w,
\end{equation}
and from~\eqref{t>1} it is evident that $\E (|\cT_1| - 1 + t)^{\ga} \leq \E (|\cT_2| - 1 + t)^{\ga}$. 
\end{proof}

\begin{remark}\label{Rsuff}
This remark concerns sufficient conditions for \eqref{Phi_order} (equivalently, by~\eqref{Phiiffy}, for $y_1 \leq y_2$). 

(a)~The condition
\begin{equation}\label{stoch_order}
\mbox{$|\cT^{(1)}| \geq |\cT^{(2)}|$ stochastically}
\end{equation}
is stronger than $y_1 \leq y_2$ and is of course equivalent to the condition that
\begin{equation}
\mbox{$\E g(|\cT^{(1)}|) \geq \E g(|\cT^{(2)}|)$}
\end{equation}
for every nonnegative nondecreasing function $g$ defined on the
positive integers.  In particular, \eqref{stoch_order} implies the
conclusions of \refT{Torder} and~\eqref{iff_ineq} 
in \refT{TLaplace}.

Note, however, that~\eqref{stoch_order} is \emph{strictly} stronger than $y_1 \leq y_2$.  While the stronger condition~\eqref{stoch_order} holds
for some of the comparisons in~\refS{S:ex} (for example, binary trees vs.\
labelled trees, for which there is monotone likelihood ratio (MLR); 
and full binary trees vs.\ ordered trees, for which there is no MLR but still
stochastic ordering),  
an example satisfying~\eqref{Phi_order} (see \refL{LPhiordt} for a proof) 
but not~\eqref{stoch_order} is $\xi_1 \sim \Po(1)$ (labelled trees) and 
$\xi_2 \sim 2 \Bi(1, \tfrac12)$ (full binary trees), because $\P(\xi_1 \leq 1) = \frac2e > \frac12 = \P(\xi_2 \leq 1)$.

(b)~Similarly, the condition
\begin{equation}\label{xi_stoch_order}
\mbox{$\xi_1 \geq \xi_2$ stochastically}
\end{equation}
is stronger than~\eqref{Phi_order}; indeed, it's even stronger than~\eqref{stoch_order}.  But this stochastic ordering of offspring distributions can only hold if~$\xi_1$ and~$\xi_2$ have the same distribution, because $\E \xi_1 = \E \xi_2 = 1$.
\end{remark}

\begin{remark}\label{Rnec}
This remark concerns necessary conditions for~\eqref{Phi_order}.


(a)
If \eqref{Phi_order} holds, then by  
a Taylor expansion near $t = 1$ \cite[(A.6)]{SJ167}
(or, alternatively, recalling~\eqref{Phiiffy}, 
by $y_1 \leq y_2$ and \cite[(A.5)]{SJ167}),
$\gs_1^2 \leq \gs_2^2$.  
(This does not require the assumption~\eqref{2+gd};
when \eqref{2+gd} holds, we can also use
\cite[Lemma 12.14]{SJ359} or~\eqref{lgdy}.)

(b)~More generally, and by similar reasoning, if~\eqref{Phi_order} holds and for some integer $r \geq 2$ we have
$\E \xi_1^j = \E \xi_2^j \leq \infty$ for $j = 1, \ldots, r - 1$, then $(-1)^r \E \xi_1^r \leq (-1)^r \E \xi_2^r \leq \infty$.
See \refApp{Ainfinite} for details.

(c) We can also consider a Taylor expansion near $t = 0$.  Thus if~\eqref{Phi_order} holds, then 
$\P(\xi_1 = 0) \leq \P(\xi_2 = 0)$.  More generally, if for some integer $r \geq 0$ we have $\P(\xi_1 = j) = \P(\xi_2 = j)$ for $j = 0, \ldots, r - 1$, then $\P(\xi_1 = r) \leq \P(\xi_2 = r)$. 
\end{remark}

We next address the question of a stronger condition than~\eqref{Phi_order}
under which the inequalities in \eqref{neg_a}--\eqref{t_shape} 
and~\eqref{iff_ineq} are all strict.  Recall the meaning of $g_1 < g_2$ described in \refD{Dineq}.

\begin{theorem}\label{Tstrict_order}
Consider two \GW{} trees, $\cT^{(1)}$ and $\cT^{(2)}$.  Suppose
\begin{equation} \label{strict_Phi_order}
\Phi_1 < \Phi_2.
\end{equation}
  \begin{thmenumerate}
  \item \label{strict_T_neg_a}
If\/ $\ga < 0$, then
\begin{equation}\label{strict_neg_a}
\mu_1(\ga) < \mu_2(\ga).
\end{equation}
  \item \label{strict_T_pos_a}
If\/ $0 < \ga < \frac12$, then
\begin{equation}\label{strict_pos_a}
\mu_1(\ga) > \mu_2(\ga).
\end{equation}
  \item \label{strict_T_shape}
We have
\begin{equation}
  \label{strict_shape}
  \mu_1' > \mu_2'.
\end{equation}
  \end{thmenumerate}
\end{theorem}

\begin{proof}
If~\eqref{strict_Phi_order} holds, then (by continuity of $\Phi_1$ and
$\Phi_2$) strict inequality $\Phi_1(t) < \Phi_2(t)$ holds over some interval
of positive length.  
The inequalities~\eqref{strict_neg_a}--\eqref{strict_shape}
are then immediate from \eqref{A.6recast}--\eqref{mu'intrecast}.
\end{proof}

\begin{theorem}\label{Tstrict_Laplace}
We have
\begin{equation}\label{strict_iff_ineq}
\mbox{\rm $\E (|\cT_1| - 1 + t)^{-m} < \E (|\cT_2| - 1 + t)^{-m}$ for all integers $m \geq 0$ and all $t \in (0, \infty)$}
\end{equation} 
if and only if~\eqref{strict_Phi_order} holds. 
\end{theorem}

\begin{proof}
The forward direction \eqref{strict_iff_ineq}$\implies$\eqref{strict_Phi_order}
follows from \refT{TLaplace}.

For the opposite direction,
use the representation~\eqref{hrep} and take derivatives with respect to
$t$.
%
\end{proof}

\begin{remark}
For all the comparison examples in \refS{S:ex} where the condition \eqref{strict_Phi_order} holds, we in fact have the stronger condition that $\Phi_1 \prec \Phi_2$.
When~\eqref{strict_Phi_order} holds, we can't have $\Phi_1 = \Phi_2$ over a nondegenerate interval because 
$\Phi_2(z) - \Phi_1(z)$ is analytic for~$z$ in the open unit disk.  But it \emph{is} possible to have $\Phi_1(t) = \Phi_2(t)$ for some values of $t \in (0, 1)$.  For an example with one such value, namely, $t = 1/6$, use the notation of \refApp{Aframework} and take $\Phi_1 = \Phi$ and $\Phi_2 = \tPhi_0$.  
\end{remark}

\subsection{Comparison examples}\label{S:ex} 

In this subsection we consider the following important examples of critical
Galton--Watson trees, and we fix the subscripting notation in
\eqref{unim}--\eqref{fullm} for the remainder of~\refS{S:comp}:
\begin{align}
\mbox{$m$-ary trees:\ }&\xi_{1, m} \sim \Bi(m, \tfrac1m)\mbox{\ ($m \geq 2$)}; \label{unim}\\
\mbox{labelled trees:\ }&\xi_2 \sim \Po(1); \label{lab}\\
\mbox{full binary trees:\ }&\xi_3 \sim 2 \Bi(1, \half); \label{fullbin}\\
\mbox{ordered trees:\ }&\xi_4 \sim \Ge(\half); \label{ord}\\
\mbox{full $m$-ary trees:\ }&\xi_{5, m} \sim m \Bi(1, \tfrac1m)\mbox{ ($m \geq 3$)}. \label{fullm}
\end{align}
Observe that
\begin{equation}\label{gs1m}
\gs_{1, m}^2 = 1 - \tfrac1m \uparrow\mbox{\ strictly as $m \uparrow$},
\end{equation}
that
\begin{equation}\label{gs5m}
\gs_{5, m}^2 = m - 1\uparrow\mbox{\ strictly as $m \uparrow$},
\end{equation}
and that, for any $m \geq 2$, we have
\begin{equation}\label{gsord}
\gs_{1, m}^2 < \gs_2^2 = \gs_3^2 < \gs_4^2 = \gs_{5, 3}^2.
\end{equation}
Further,
\begin{equation}\label{thirdord23}
\E \xi_2^3 = 5 > 4 = \E \xi_3^3
\end{equation}
and
\begin{equation}\label{thirdord45}
\E \xi_4^3 = 13 > 9 = \E \xi_{5, 3}^3.
\end{equation}
According to \refR{Rnec}(a)--(b) and \eqref{gs1m}--\eqref{thirdord45}, the only possible $\Phi$-orderings in the order $<$ among the trees listed in \eqref{unim}--\eqref{fullm} are
\begin{equation}\label{Phi1m}
\Phi_{1, m} \uparrow\mbox{\ strictly as $m \uparrow$},
\end{equation}
\begin{equation}\label{Phi5m}
\Phi_{5, m} \uparrow\mbox{\ strictly as $m \uparrow$},
\end{equation}
and, for any $m \geq 2$,
\begin{equation}\label{Phiord}
\Phi_{1, m} < \Phi_2 < \Phi_3 < \Phi_4 < \Phi_{5, 3}.
\end{equation}

Alternatively, we can note that
\begin{equation}
\P(\xi_{1, m} = 0) = (1 - \tfrac1m)^m \uparrow\mbox{\ strictly as $m \uparrow$}
\end{equation}
(see~\eqref{logs} below with $t = 0$); that
\begin{equation}
\P(\xi_{5, m} = 0) = 1 - \tfrac1m \uparrow\mbox{\ strictly as $m \uparrow$};
\end{equation}
that, for any $m \geq 2$, we have
\begin{equation}
\P(\xi_{1, m} = 0) < e^{-1} = \P(\xi_2 = 0) < \P(\xi_3 = 0) = \P(\xi_4 = 0) < \P(\xi_{5, 3} = 0);
\end{equation}
and, further, that
\begin{equation}
\P(\xi_3 \leq 1) = \half < \tfrac34 = \P(\xi_4 \leq 1)
\end{equation}
to conclude again, now using \refR{Rnec}(c), that the only possible $\Phi$-orderings in the order $<$ for \eqref{unim}--\eqref{fullm} are \eqref{Phi1m}--\eqref{Phiord}.
 
Remarkably, all the inequalities in \eqref{Phi1m}--\eqref{Phiord} are true, and in fact there is strict inequality at \emph{every} argument.
\begin{lemma}\label{LPhiordt}
For every $t \in (0, 1)$ we have
\begin{equation}\label{Phi1mt}
\Phi_{1, m}(t) \uparrow\mbox{\rm \ strictly as $m \uparrow$},
\end{equation}
\begin{equation}\label{Phi5mt}
\Phi_{5, m}(t) \uparrow\mbox{\rm \ strictly as $m \uparrow$};
\end{equation}
and, for any $m \geq 2$,
\begin{equation}\label{Phiordt}
\Phi_{1, m} \prec \Phi_2 \prec \Phi_3 \prec \Phi_4 \prec \Phi_{5, 3}.
\end{equation}
\end{lemma}

\begin{proof}
The proof is a collection of simple exercises in calculus.
\medskip

\emph{Proof of~\eqref{Phi1mt}}.  Fix $m \geq 2$ and $t \in (0, 1)$.  Observe that
\begin{equation}\label{Phi1mform}
\Phi_{1, m}(t) = (\tfrac{m - 1}{m} + \tfrac{1}{m} t)^m = [1 - \tfrac1m (1 - t)]^m. 
\end{equation}
Thus
\begin{align}
\lefteqn{\log \Phi_{1, m + 1}(t) - \log \Phi_{1, m}(t)} 
\notag\\ 
&= (m + 1) \log\Bigpar{1 - \frac{1 - t}{m + 1}} - m \log\Bigpar{1 - \frac{1 - t}{m}} \notag\\
&= - \left[ (1 - t) + \frac{(1 - t)^2}{2 (m + 1)} + \frac{(1 - t)^3}{3 (m + 1)^2} + \cdots \right]
+ \left[ (1 - t) + \frac{(1 - t)^2}{2 m} + \frac{(1 - t)^3}{3 m^2} + \cdots \right] \notag\\
&> 0. \label{logs}
\end{align}
\smallskip

\emph{Proof of~\eqref{Phi5mt}}.  Fix $m \geq 3$.  Consider $t \in (0, 1]$ and observe that
\begin{equation}\label{Phi5mform}
\Phi_{5, m}(t) = \tfrac1m (m - 1 + t^m) 
\end{equation}
Let $f(t) := \Phi_{5, m + 1}(t) - \Phi_{5, m}(t)$.  We have $f(1) = 1 - 1 = 0$ and
\begin{align}
f'(t) = t^m - t^{m - 1} = - t^{m - 1} (1 - t) < 0
\end{align}
for $t \in (0, 1)$.  Thus $f(t) > 0$ for $t \in (0, 1)$.
\medskip

\emph{Proof of $\Phi_{1, m} \prec \Phi_2$ for $2 \leq m < \infty$}. From~\eqref{Phi1mform} we see that
\begin{align}
\Phi_{1, \infty}(t) := \lim_{m \to \infty} \Phi_{1, m}(t) = e^{t - 1} = \Phi_2(t).
\end{align}
The result follows.
\medskip

\emph{Proof of $\Phi_2 \prec \Phi_3$}. Consider $t \in (0, 1]$ and let 
\begin{align}
f(t) := \ln \Phi_3(t) - \ln \Phi_2(t) = \ln(1 + t^2) - \ln 2 - (t - 1).
\end{align}
We have $f(1) = 0$ and
\begin{align}
f'(t) = 2 t (1 + t^2)^{-1} - 1 = - (1 - t)^2 (1 + t^2)^{-1} < 0
\end{align}
for $t \in (0, 1)$.  Thus $f(t) > 0$ for $t \in (0, 1)$.
\medskip

\emph{Proof of $\Phi_3 \prec \Phi_4$}.  Consider 
$t \in [0, 1]$ and let
\begin{align}
f(t) := \Phi_4(t) - \Phi_3(t) = \half (1 - \half t)^{-1} - \half (1 + t^2) = \tfrac14 t (1 - t)^2 (1- \half t)^{-1}.
\end{align}
Clearly, $f(t) > 0$ for $t \in (0, 1)$. 
\medskip

\emph{Proof of $\Phi_4 \prec \Phi_{5, 3}$}.  Consider $t \in (0, 1)$ and let
\begin{align}
f(t) 
:= \frac{\Phi_{5, 3}(t)}{\Phi_4(t)} = \frac{\tfrac23 + \tfrac13 t^3}{\half (1 - \half t)^{-1}}.
\end{align}
Then
\begin{align}
f(t) = 1 + \tfrac13 (1 - 2 t + 2 t^3 - t^4) = 1 + \tfrac13 (1 - t)^3 (1 + t) > 1,
\end{align} 
as desired.
\end{proof}

\begin{theorem}
  \begin{thmenumerate}
  \item \label{muord_neg_a}
If $\ga < 0$, then
\begin{equation}\label{mu1m_neg}
\mu_{1, m}(\ga) \uparrow\mbox{\rm \ strictly as $m \uparrow$},
\end{equation}
\begin{equation}\label{mu5mt_neg}
\mu_{5, m}(\ga) \uparrow\mbox{\rm \ strictly as $m \uparrow$};
\end{equation}
and, for any $m \geq 2$,
\begin{equation}\label{muord_neg}
\mu_{1, m}(\ga) < \mu_2(\ga) < \mu_3(\ga) < \mu_4(\ga) < \mu_{5, 3}(\ga).
\end{equation}
  \item \label{muord_pos_a}
The orders in~\ref{muord_neg_a} are all reversed for $0 < \ga < \half$ and for $\mu'$. 
  \end{thmenumerate}
\end{theorem}

\begin{proof}
The theorem is immediate from \refL{LPhiordt} and 
\refT{Tstrict_order}.
\end{proof}

\begin{remark}\label{Rxileq2}
The only two examples among \eqref{unim}--\eqref{fullm} for which $\xi \leq 2$ \as\ are binary trees with 
$\Phi_{1, 2}(t) = \frac14(1 + t)^2$ and full binary trees for which $\Phi_3(t) = \half (1 + t^2)$.  These are two examples ($c = \half$ and $c = 1$, respectively) of the most general critical \GW\ offspring distribution $\xi\cx$ to satisfy $\xi\cx \leq 2$ \as, with $0 < c \leq 1$ and
\begin{align}
\P(\xi\cx = 0) = \P(\xi\cx = 2) = \half c, \qquad \P(\xi\cx = 1) = 1 - c.
\end{align} 
Generalizing $\Phi_{1, 2} \prec \Phi_3$ from~\eqref{Phiordt} in \refL{LPhiordt}, we claim that $\Phi\cx$ is strictly increasing in the order~$\prec$.  Indeed, for $t \in (0, 1)$ we have
\begin{equation}\label{Phic}
\Phi\cx(t) = t + \half c (1 - t)^2,
\end{equation}
which is clearly strictly increasing in $c \in (0, 1]$. 
\end{remark}

\begin{remark}
Despite a suggestion to the contrary provided by \refL{LPhiordt} and \refR{Rxileq2}, the partial order $\leq$ on tree-size {\pgf}s is \emph{not} a linear order.  An example of incomparable $\Phi$ and $\tPhi$ is provided in \refApp{Aframework} (taking $\eps \in (0, 1]$ in the notation there).
For a simpler counterexample, which shows that
$\leq$ does not even linearly order cubic \pgf{s},  
let
\begin{align}
\Phi(t) := \Phi_{1, 2}(t) = \tfrac14 (1 + t)^2
\end{align}
correspond to binary trees, as at~\eqref{unim}; and let
\begin{align}
\tPhi(t) := \tfrac{6}{32} + \tfrac{23}{32} t + \tfrac{3}{32} t^3.
\end{align}
If $t_0 = \half$, then
\begin{align}
\Phi(t_0) = \tfrac{9}{16} = \tfrac{144}{256} > \tfrac{143}{256} = \tPhi(t_0);
\end{align} 
while if $t_1 = \frac78$, then
\begin{align}
\Phi(t_1) = \tfrac{225}{256} = \tfrac{14400}{16384} < \tfrac{14405}{16384} = \tPhi(t_1).
\end{align}
\end{remark}

\subsection{Numerical computation of~$\mu'$}\label{S:num}

In this subsection we will compute the constant $\mu'$ for several examples of critical \GW\ trees.  First, to set the stage for what to expect, we consider in the next remark the possible values of $\mu'$ as $\xi$ ranges over all critical offspring distributions.

Recall~\eqref{mu'intrecast}.  For the next remark, we find it convenient to
break the integral into two pieces,
using the notation $x^+:=\max\set{x,0}$ and $x^-:=\max\set{-x,0}$:
\begin{align}\label{mu'split}
\mu' 
&= - \gam - \int_{t \in (0, 1)} [\log \log R(t)]^+ \dd t 
+ \int_{t \in (0, 1)} [\log \log R(t)]^- \dd t \nonumber\\
&= - \gamma - J_+ + J_-, 
\end{align}
say.

\begin{remark}\label{Rxileq2y}
In this remark we argue that there is no finite upper bound, nor positive lower bound, on $\mu'$ over all \GW\ trees.

(a)~Referring to \refR{Rxileq2}, observe that
\begin{equation}
\Phi\cx(t) \downto t
\end{equation}
for each $t \in (0, 1)$ as $c \downto 0$.  By the dominated convergence theorem (DCT), $J_+ \downto 0$ as $c \downto 0$.  By the monotone convergence theorem (MCT), $J_- \upto \infty$ as $c \downto 0$.  Thus, as $c \downto 0$ we have  
\begin{equation}
\mu'\cx \upto \infty.
\end{equation}
Indeed, it can be shown that $\mu'_{(c)} = \log \frac2c + 1 - \gamma + o(1)$ as $c \to 0$.

(b)~For
the offspring distributions $\xi_{5,m}$, we have as \mtoo\ that
\begin{align}
  \P(\xi_{5,m}=0)=1-\tfrac{1}m\to1,
\end{align}
and thus $\xi_{5, m} \pto 0$, which implies convergence of the \pgf{s} for every
$t\in(0,1)$; hence,
\begin{equation}
\Phi_{5, \infty}(t) := \lim_{m \to \infty} \Phi_{5, m}(t) = 1,
\end{equation}
which is otherwise obvious by direct calculation (showing also that the limit is an increasing one).
By the MCT applied to $J_+$ and the DCT applied to $J_-$, we find
\begin{equation}
\mu'_{5, m} \downto - \gamma - \int_0^1\!(\log \log \tfrac1t) \dd t = 0
\end{equation}
as $m \upto \infty$.  Indeed, it can be shown that $\mu'_{5, m} \sim m^{-1} \ln m$ as $m \to \infty$.

(c) We claim that the image of $\mu'$ over \GW\ tree models is in fact $(0, \infty)$.  To see this, we first note 
that $\mu'\cx$ is continuous in~$c$, 
with $\mu'\cxx1=\mu'_3$,
which by (a)
implies that the image contains $[\mu'_3, \infty)$.  Further, by considering the offspring {\pgf}s
\begin{equation}
\Phi_{\gl, m} := \gl \Phi_{5, m} + (1 - \gl) \Phi_3
\end{equation}
with $\gl \in [0, 1]$, one can show (by consideration of large~$m$) that the image also contains $(0, \mu'_3)$; we omit the details.

(d) Similarly as for~(c), for each fixed value of $\ga < 0$ the image of $\mu(\ga)$ over all \GW\ tree models is 
$(0, 1)$, and for each fixed value of $\ga \in (0, \half)$ the image is $(1, \infty)$. 
\end{remark}

\begin{example}\label{Eexplicitmu'}
The constant $\mu'_{1, 2}$ is computed to 50 digits in \cite[Section~5.2]{FillFK} using the alternative form
\begin{equation}
\mu' = - \gamma - \int_0^1\!(\log \log \tfrac1t) y'(t) \dd t
\end{equation}
of~\eqref{mu'intrecast}, explicit calculation of
\begin{equation}
y_{1, 2}(t) = \frac{2 - t - 2 \sqrt{1 - t}}{t} = t (1 + \sqrt{1 - t})^{-2}
\end{equation}
and thence its derivative
\begin{equation}
y_{1, 2}'(t) =  (1 - t)^{-1/2} (1 + \sqrt{1 - t})^{-2},
\end{equation}
and numerical integration.  But it is easier to use~\eqref{mu'intrecast} for (high-precision) computation of $\mu'$, especially for the values $\mu'_{1, m}$ and $\mu'_{5, m}$.

As examples, we find, rounded to five digits,
\begin{equation}
\mu'_{1, 2} = 2.0254, \qquad \mu'_2 = 1.5561, \qquad \mu'_3 = 1.4414, \qquad \mu'_4 = 1.1581.
\end{equation}
Note that
\begin{align}
\infty > \mu'_{1, 2} > \mu'_2 > \mu'_3 > \mu'_4 > 0,
\end{align}
as guaranteed by \refL{LPhiordt} and
\refT{Tstrict_order}\ref{strict_T_shape}; 
see also \refR{Rxileq2y} concerning the \emph{a priori} lack of an upper bound on $\mu'_{1, 2}$ and a positive lower bound on $\mu'_4$.

As other examples, we find, rounded to five digits, 
\begin{equation}
\mu'_{1, 2} = 2.0254, \qquad \mu'_{1, 3} = 1.8224, \qquad \mu'_{1,10^3} = 1.5567;
\end{equation}
\begin{equation}
\mu'_{5, 3} = 1.0164, \qquad \mu'_{5, 4} = 0.80800, \qquad \mu'_{5, 10^6} = 1.5372 \times 10^{-5};
\end{equation}
and, in the notation of \refR{Rxileq2},
\begin{equation}
\mu'_{(10^{-6})} = 14.931, \qquad \mu'_{(1/2)} = \mu'_{1, 2} = 2.0254, \qquad \mu'_{(1 - 10^{-2})} = 1.4496;
\end{equation}
\end{example}

\appendix

\ccreset
\section{Comparison counterexamples}\label{ACC}

\subsection{A framework for comparison counterexamples}\label{Aframework}

In this \refApp{Aframework} we establish a framework for various counterexamples involving comparisons of offspring distributions in \refS{S:comp}.  The idea is to set up two offspring distributions, say~$\xi$ and~$\txi$, with respective 
\pgf{s}~$\Phi$ and~$\tPhi$, such that, for (real) $t \in [0, 1)$, the difference $\gD(t) := \tPhi(t) - \Phi(t)$ satisfies $\gD(t) > 0$ for most values of~$t$, and $\gD(t) \leq 0$ (but not by much) for~$t$ very near $\frac16$ (with this value somewhat arbitrarily chosen).

Let $\xi$ have the following probability mass function 
satisfying $\E \xi = 1$, as required for a critical offspring distribution:
\begin{align}
p_0 &:= \P(\xi = 0) = \tfrac14 + 3 e^{-3} + 5 e^{-11} > 0, \\
p_1 &:= \P(\xi = 1) = \half > 0, \\
p_2 &:= \P(\xi = 2) = \tfrac14 - 4 e^{-3} + 36 e^{-11} > 0, \\
p_k &:= \P(\xi = k) = e^{-11} \frac{8^k}{k!} > 0\mbox{\ for $k \geq 3$}.
\end{align}
We denote its \pgf\ by~$\Phi$.

Let~$\eps \geq 0$ and for $t \in [0, 1]$ 
define
\begin{equation}
g_{\eps}(t) := 
\half \left[ 1 - \cos\bigpar{(4 \pi)(\tfrac35 t + \tfrac25)} \right] - \eps (1 - t)^3. 
\end{equation}
Note that
\begin{equation}\label{geps0}
g_{\eps}(1) = g_{\eps}'(1) = 0; 
\end{equation}
moreover, for every $t \in [0, 1]$ we have 
\begin{equation}\label{gbounds}
- \eps \leq g_{\eps}(t) \leq 1
\end{equation}
(in particular, $g_0(t) \geq 0$), 
and one can verify for 
small $\eps > 0$ that the set $\{t \in [0, 1): g_\eps(t) < 0\}$ is an open interval of length $O(\eps^{1/2})$ 
containing~$\frac16$.

Because $12 \pi / 5 < 8$, it's easy to check that there exists 
$\cc > 0$ such that for all 
$\eps \in [0, 1]$ the function
\begin{equation}\label{tPhieps}
\tPhi_{\eps}(t) := \Phi(t) + \ccx g_{\eps}(t)
\end{equation}
has a power series expansion about the origin with nonnegative coefficients.  From \eqref{geps0} it now follows that $\tPhi_{\eps}$ is the \pgf\ of a random variable~$\txi$ with $\E \txi = 1$.

As 
we have now discussed, the difference function $\gD_{\eps}(t) := \tPhi_{\eps}(t) - \Phi(t)$ is nonnegative when 
$\eps = 0$; and when $\eps > 0$ is small, the set 
\begin{equation}\label{interval}
I_{\eps} := 
\{t \in [0, 1): \gD_\eps(t) < 0\}=
\{t \in [0, 1): g_\eps(t) < 0\}
\end{equation}
is an open interval of length $O(\eps^{1/2})$ containing $\frac16$.

Although  
not needed anywhere in \refS{S:comp} nor in this \refApp{ACC}, we note in passing that both $\xi$ and $\txi_{\eps}$ have \mgf{s} that are finite everywhere and \pgf{s} that are entire; in particular, both satisfy~\eqref{2+gd}. 

\subsection{The converse to \refT{Torder} fails}\label{Aconverse}
In this subsection we show that~\ref{T_neg_a} and \ref{T_pos_a} of \refT{Torder} together do not imply~\eqref{Phi_order}.  In fact, not even the strict inequalities in~\ref{strict_T_neg_a}--\ref{strict_T_shape} of \refT{Tstrict_order} do.

\begin{example}
In the notation of \refApp{Aframework}, take $\Phi_1$ to be~$\Phi$ and
$\Phi_2$ to be the \pgf~$\tPhi_{\eps}$ of~\eqref{tPhieps}.  We do not have $\Phi_1 \leq \Phi_2$.  But we claim that for all sufficiently small $\eps > 0$ (not depending on~$\ga$, to be clear), 
\refT{Tstrict_order}\ref{strict_T_neg_a}--\ref{strict_T_shape} hold.

To establish the desired inequalities about $\mu(\ga)$, we will utilize~\eqref{A.6recast}.  
For this we apply the mean value theorem to the function $x \mapsto (\log x)^{-\ga}$, $x \in (1, \infty)$, as follows.  Let $1 < x_1 \leq x_2 < \infty$.  If $\ga \leq -1$, then for some point $x \in [x_1, x_2]$ we have
\begin{samepage}
\begin{align}\label{hw1}
\lefteqn{\hspace{-.1in}(\log x_2)^{-\ga} - (\log x_1)^{-\ga}} \nonumber \\
&= (- \ga) x^{-1} (\log x)^{ - \ga - 1} (x_2 - x_1) \nonumber \\ 
&\in [(- \ga) x_2^{-1} (\log x_1)^{ - \ga - 1} (x_2 - x_1),\ (- \ga) x_1^{-1} (\log x_2)^{ - \ga - 1} (x_2 - x_1)].
\end{align}
\end{samepage}%
Similarly, if $\ga \in (-1, 0)$, then
\begin{align}\label{hw2}
\lefteqn{\hspace{-.1in}(\log x_2)^{-\ga} - (\log x_1)^{-\ga}} \nonumber \\ 
&\in [(- \ga) x_2^{-1} (\log x_2)^{ - \ga - 1} (x_2 - x_1),\ (- \ga) x_1^{-1} (\log x_1)^{ - \ga - 1} (x_2 - x_1)]; 
\end{align}
and if $\ga > 0$, then
\begin{align}\label{hw3}
\lefteqn{\hspace{-.2in}(\log x_1)^{-\ga} - (\log x_2)^{-\ga}} \nonumber \\
&\in [\ga x_2^{-1} (\log x_2)^{ - \ga - 1} (x_2 - x_1),\ \ga x_1^{-1} (\log x_1)^{ - \ga - 1} (x_2 - x_1)].
\end{align}

For $t \in (0, 1) \setminus I_{\eps}$ we have 
$\tPhi_\eps(t) \ge \Phi(t)$ and 
thus
$\tR_\eps(t) \ge R(t)$; hence
it follows from \eqref{hw1}--\eqref{hw3} that 
\begin{align}
[\log \tR_{\eps}(t)]^{-\ga} - [\log R(t)]^{-\ga}
&\geq |\ga| \frac{1}{\tR_{\eps}(t)} [\log R(t)]^{ - \ga - 1} \frac{\ccx g_{\eps}(t)}{t} 
&&\mbox{if $\ga \leq -1$}; 
\label{mostsmall} \\ 
[\log \tR_{\eps}(t)]^{-\ga} - [\log R(t)]^{-\ga}
&\geq |\ga| \frac{1}{\tR_{\eps}(t)} [ \log \tR_{\eps}(t)]^{ - \ga - 1} \frac{\ccx g_{\eps}(t)}{t} 
&&\mbox{if $\ga \in (-1, 0)$}; 
\label{mostmed} \\
[\log R(t)]^{-\ga} - [\log \tR_{\eps}(t)]^{-\ga}
&\geq \ga \frac{1}{\tR_{\eps}(t)} [\log \tR_{\eps}(t)]^{ - \ga - 1} \frac{\ccx g_{\eps}(t)}{t} 
&&\mbox{if $\ga > 0$}.
\label{mostbig} 
\end{align}
Denote the interval~$I_{\eps}$ defined at~\eqref{interval} by $(a_{\eps}, b_{\eps})$.
Consider $t \in I_{\eps}$ for the next three displays;
thus $\Phi(t) > \tPhi_\eps(t)$ and $R(t) > \tR_\eps(t)$.
If $\ga \leq -1$ we have, recalling \refL{Lstrict},
\begin{align}
[\log \tR_{\eps}(t)]^{-\ga} - [\log R(t)]^{-\ga}
&\geq - |\ga| \frac{1}{\tR_\eps(t)} [\log R(t)]^{ - \ga - 1} \frac{\ccx |g_{\eps}(t)|}{t} 
\nonumber \\
&= - |\ga| \frac{1}{\tPhi_\eps(t)} [\log R(t)]^{ - \ga - 1} \ccx |g_{\eps}(t)| 
\nonumber \\
&\geq - |\ga| \frac{1}{\tPhi_\eps(a_{\eps})} [\log R(a_{\eps})]^{ - \ga - 1} \ccx \eps,
\label{intervalsmall}
\end{align}
where we have used~\eqref{gbounds} at the last inequality; similarly, if $\ga \in (-1, 0)$ we have 
\begin{align} 
[\log \tR_{\eps}(t)]^{-\ga} - [\log R(t)]^{-\ga}
&\geq - |\ga| \frac{1}{\tPhi_\eps(t)} [\log \tR_{\eps}(t)]^{ - \ga - 1} \ccx |g_{\eps}(t)| \nonumber \\
&\geq - |\ga| \frac{1}{\tPhi_\eps(a_{\eps})} [\log \tR_{\eps}(b_{\eps})]^{ - \ga - 1} \ccx \eps;
\label{intervalmed}
\end{align}
and if $\ga > 0$ we have
\begin{align}
[\log R(t)]^{-\ga} - [\log \tR_{\eps}(t)]^{-\ga}
&\geq - \ga \frac{1}{\tPhi_\eps(t)} [\log \tR_{\eps}(t)]^{ - \ga - 1} \ccx |g_{\eps}(t)| \nonumber \\
&\geq - \ga \frac{1}{\tPhi_\eps(a_{\eps})} [\log \tR_{\eps}(b_{\eps})]^{ - \ga - 1} \ccx \eps.
\label{intervalbig}
\end{align}

We continue by assessing the contribution
to the difference $\sgn(\ga) \cdot [\mu(\ga) - \tmu_{\eps}(\ga)]$ of integrals from $t \in I_{\eps}$, with asserted inequalities valid for all small $\eps > 0$. 
If $\ga \leq -1$, the contribution is, using~\eqref{intervalsmall}, bounded
below by
\begin{align}
- |\ga| \ccx \eps \frac{1}{\tPhi_\eps(a_{\eps})} [\log R(a_{\eps})]^{ - \ga - 1} 
(b_{\eps} - a_{\eps}) 
&\geq - |\ga| \CC \eps^{3/2} \bigsqpar{\log \bigpar{R(1/6) + \CC \eps^{1/2}}}^{ - \ga - 1} \nonumber \\
&\geq - |\ga| C_1 \eps^{3/2} [\log R(1/7)]^{ - \ga - 1},
\label{negsmall}
\end{align}
where we have used \refL{Lstrict} and where the constants $C_1$  and $C_2$ do not depend on~$\ga$.
Similarly, for $\ga \in (-1, \half)$ the contribution is bounded below by
\begin{equation}\label{negnotsmall}
- |\ga| \CC \eps^{3/2} [\log R(1/5)]^{ - \ga - 1}. 
\end{equation}

Next we similarly assess the contribution
to $\sgn(\ga) \cdot [\mu(\ga) - \tmu_{\eps}(\ga)]$ from values $t \in (0, 1) \setminus I_{\eps}$.  For all small $\eps > 0$, for $\ga \leq -1$ the contribution is, using~\eqref{mostsmall}, at least
\begin{equation}
\int_{1/9}^{1/8}\!|\ga| \frac{1}{\tPhi_{\eps}(t)} [\log R(t)]^{ - \ga - 1} \ccx g_{\eps}(t) \dd t
\geq |\ga|\,\cc\,[\log R(1/8)]^{ - \ga - 1},
\label{possmall}
\end{equation}
where the constant $\ccx$ does not depend on~$\ga$.
Similarly, for $\ga \in (-1, 1/2)$ the contribution is, 
using~\eqref{mostmed}--\eqref{mostbig}, at least
\begin{equation}
|\ga| \cc \int_{7/24}^{1/3}\![\log \tR_{\eps}(t)]^{ - \ga - 1} \dd t
\geq |\ga| \cc\, [\log R(1/4)]^{ - \ga - 1}.
\label{posnotsmall}
\end{equation}

Summarizing, for $\ga \leq -1$ we have, using~\eqref{possmall} and~\eqref{negsmall},
\begin{equation}
\tmu_{\eps}(\ga) - \mu(\ga) 
\geq |\ga| c_2 [\log R(1/8)]^{ - \ga - 1}
- |\ga| C_1 \eps^{3/2} [\log R(1/7)]^{ - \ga - 1}; 
\end{equation}
and for $\ga \in (-1, \half)$ we have, using~\eqref{posnotsmall}
and~\eqref{negnotsmall},
\begin{align}
\sgn(\ga) \cdot [\mu(\ga) - \tmu_{\eps}(\ga)] 
&\geq |\ga| c_4 [\log R(1/4)]^{ - \ga - 1}
- |\ga| C_3 \eps^{3/2} [\log R(1/5)]^{ - \ga - 1}.
\end{align}
Since by \refL{Lstrict}
\begin{align}
R(1/8) > R(1/7)
\mbox{\quad and\quad}R(1/4) < R(1/5),
\end{align}
for sufficiently small $\eps \leq (\min\{c_2 / C_1, c_4 / C_3\})^{2/3}$ the desired strict inequalities all follow.  Our calculations also  demonstrate that
\begin{equation}
\mu' - \tmu'_{\eps} 
\geq c_4 [\log R(1/4)]^{- 1} 
- C_3 \eps^{3/2} [\log R(1/5)]^{ - 1},  
\end{equation}
which is (strictly) positive for sufficiently small $\eps \leq (c_4 / C_3)^{2/3}$. 
\end{example}

\section{Negative moments of affine functions of tree size}\label{Amono}
The representation~\eqref{A.6recast} of $\mu(\ga)$ as an integral in terms of the offspring \pgf~$\Phi$ and the consequent ordering of $\mu$-values exhibited in \refT{Torder}\ref{T_neg_a} can be extended to treat means of more general functions of the \GW\ tree-size.  We illustrate this with the following theorem, used in the proof of \refT{TLaplace}.

\begin{theorem}\label{Tt>1}
For real $\ga < 0$ and $t > 1$, we have
\begin{equation}\label{t>1app}
\E (|\cT| - 1 + t)^{\ga} = (t - 1)^{\ga} \int_0^1\!\left[ 1 - \frac{c(\eta; - \ga, t)}{\Gamma(-\ga)} \right] \dd \eta, 
\end{equation}
where $c(\eta; - \ga, t)$ is the incomplete gamma function value
\begin{equation}
c(\eta; -\ga, t) = \int_{(t - 1) \log R(\eta)}^{\infty}\!v^{ - \ga - 1} e^{-v} \dd v.
\end{equation}
\end{theorem}

\begin{proof}
Let $f(s) := (s - 1 + t)^{\ga}$. 
Observe that $s \mapsto f(s) / s$ for $s > 0$ is the Laplace transform of the (strictly) increasing function~$g$ mapping $x > 0$ to
\begin{equation}\label{b3}
g(x) := (t - 1)^{\ga} \left[ 1 - \frac{\gamma((t - 1) x; -\ga)}{\Gamma(-\ga)} \right] \in (0, (t - 1)^{\alpha}),
\end{equation}
where here $\gamma(\cdot; -\ga)$ is the incomplete gamma function
\begin{equation}
\int_{\cdot}^\infty\!v^{ - \ga - 1} e^{-v} \dd v.
\end{equation}
Therefore
\begin{align}
\E (|\cT| - 1 + t)^{\ga}
&= \E f(|\cT|) \nonumber \\
&= \sum_{n = 1}^{\infty} \P(|\cT| = n) f(n) 
= \sum_{n = 1}^{\infty} n \P(|\cT| = n) \int_0^{\infty} e^{ - n x} g(x) \dd x \nonumber \\
&= \int_0^{\infty}\!g(x) y'(e^{-x}) e^{-x} \dd x 
= \int_0^1\!g( - \log u) y'(u) \dd u \nonumber \\
&= \int_0^1\!g\bigpar{\log (R(\eta))} \dd \eta,
\end{align}
again changing variables by $u = y\qw(\eta) = 1 / R(\eta)$, 
and \eqref{t>1app} follows by \eqref{b3}.
\end{proof}

\section{Comparisons allowing infinite offspring moments}\label{Ainfinite}
\refR{Rnec}(b) follows quickly from the following theorem concerning Laplace transforms.

\begin{theorem}\label{TLaplacetrans}
Let~$\xi$ 
be a (not necessarily integer-valued) nonnegative random variable
with Laplace transform~$f$ and moments
\begin{equation}
m_j := \E \xi^j\le\infty, \quad j = 0, 1, 2, \dots
\end{equation}
(with $m_0 := 0$).  For a given positive integer~$r$, suppose that $m_{r - 1} < \infty$.
Then
\begin{equation}
g(t) := (-1)^r t^{-r} r! \left[ f(t) - \sum_{j = 0}^{r - 1} (-1)^j m_j \frac{t^j}{j!} \right]
\end{equation}
is nonnegative 
for $t > 0$ and increases (weakly) to $m_r \leq \infty$ as $t \downto 0$. 
\end{theorem} 

We will prove \refT{TLaplacetrans} using the following calculus lemma.

\begin{lemma}\label{Lexpcalc}
Let~$r$ be a fixed positive integer, and define
\begin{equation}
h(x) := (-1)^r x^{-r} \left[ e^{-x} - \sum_{j = 0}^{r - 1} (-1)^j \frac{x^j}{j!} \right], \quad x > 0.
\end{equation}
Then~$h$ is (strictly) positive and (strictly) decreasing, with limit $1 / r!$ as $x \downto 0$.
\end{lemma}

\begin{proof}
The lemma is immediate from the claim that
\begin{equation}\label{Taylor}
h(x) = \frac{1}{(r - 1)!} \int_0^1\!v^{r - 1} e^{- x (1 - v)} \dd v. 
\end{equation}
We offer two proofs of this claim.

\begin{proof}[Proof \#1 of~\eqref{Taylor}]
By Taylor's theorem with remainder in integral form,
\begin{equation}
h(x) = \frac{x^{-r}}{(r - 1)!} \int_0^x\!(x - u)^{r - 1} e^{-u} \dd u.
\end{equation}
Now simply change the variable of integration from~$u$ to $v = 1 - \frac{u}{x}$.
\noqed
\end{proof}

\begin{proof}[Proof \#2 of~\eqref{Taylor}]
Let~$B$ denote Euler's beta function.  Then the right side of~\eqref{Taylor} equals
\begin{align}
\frac{1}{(r - 1)!} \sum_{k = 0}^{\infty} (-1)^k \frac{x^k}{k!} \int_0^1\! v^{r - 1} (1 - v)^k \dd v
&= \frac{1}{(r - 1)!} \sum_{k = 0}^{\infty} (-1)^k \frac{x^k}{k!} B(r, k + 1)
\notag\\
&= (-1)^r x^{-r} \sum_{j = r}^{\infty} (-1)^j \frac{x^j}{j!} = h(x).
\end{align}
\noqed
\end{proof}
\end{proof}

\begin{proof}[Proof of \refT{TLaplacetrans}]
For $t > 0$ we have
\begin{align}
g(t) 
&= \E \left[ (-1)^r t^{-r} r! \left( e^{- t \xi} - \sum_{j = 0}^{r - 1} (-1)^j \frac{(t \xi)^j}{j!} \right) \right] \notag\\
&= \E \left[ (-1)^r t^{-r} r! \left( e^{- t \xi} - \sum_{j = 0}^{r - 1} (-1)^j \frac{(t \xi)^j}{j!} \right);\,\xi > 0 \right] \notag\\
&= \E \left[ r!\,h(t \xi)\,\xi^r;\,\xi > 0 \right].
\end{align}
By \refL{Lexpcalc}, the nonnegative random variables $r!\,h(t \xi)\,\xi^r {\bf 1}(\xi > 0)$ increase (weakly) to 
$\xi^r {\bf 1}(\xi > 0) = \xi^r$ as $t \downto 0$.  Thus, by the MCT, $g(t) \upto m_r \leq \infty$ as $t \downto 0$.    
\end{proof}

\section{Corrigendum to \cite{SJ359}}\label{Atypo}
As said in \refR{R=}, there is a typo in \cite[Theorem D.1]{SJ359}; the
variance given in (D.2) there is incorrect and should be
\begin{align*}
    \E|\zeta|^2=
\frac{1}{2\sqrt\pi} \Re\frac{\gG(\ii t-\frac12)}{\gG(\ii t)},
\end{align*}
as stated in \eqref{ri2}. 

The formula (D.2) in \cite{SJ359} has, incorrectly, $\gG(\ii t-1)$ in the
denominator, which comes from (D.5) which has the same error.
Formula (D.8) in the proof is correct, with denominator $\gG(\ii t)$,
and yields (D.5) and (D.2) with the same denominator, 
i.e., \eqref{ri2}.

Theorem D.1 in \cite{SJ359} also claims that $\E|\zeta|^2>0$. 
The proof is based on the incorrect formula given there, but luckily the
same proof applies also to the correct formula. In (D.14) we
obtain $\gG(1-\ii t)$ instead of $\gG(2-\ii t)$ (and an immaterial
change of sign); hence we have to show that 
$\gG(1-\ii t)/\gG(\frac32-\ii t)$ is not real for $t\neq0$.
Thus, in (D.15), we should have $-\Im \int_{1}^{3/2}\psi(s-\ii t)\dd s$.
We use (D.18) as before, and now see that if $t<0$, then
$0>\arg\bigpar{\gG(1-\ii t)/\gG(\frac32-\ii t)}>-\pi/4$, which completes the
proof that the variance in \eqref{ri2} (and the display above) is nonzero.


\newcommand\AAP{\emph{Adv. Appl. Probab.} }
\newcommand\JAP{\emph{J. Appl. Probab.} }
\newcommand\JAMS{\emph{J. \AMS} }
\newcommand\MAMS{\emph{Memoirs \AMS} }
\newcommand\PAMS{\emph{Proc. \AMS} }
\newcommand\TAMS{\emph{Trans. \AMS} }
\newcommand\AnnMS{\emph{Ann. Math. Statist.} }
\newcommand\AnnPr{\emph{Ann. Probab.} }
\newcommand\CPC{\emph{Combin. Probab. Comput.} }
\newcommand\JMAA{\emph{J. Math. Anal. Appl.} }
\newcommand\RSA{\emph{Random Struct. Alg.} }
\newcommand\ZW{\emph{Z. Wahrsch. Verw. Gebiete} }
\newcommand\DMTCS{\jour{Discr. Math. Theor. Comput. Sci.} }

\newcommand\AMS{Amer. Math. Soc.}
\newcommand\Springer{Springer-Verlag}
\newcommand\Wiley{Wiley}

\newcommand\vol{\textbf}
\newcommand\jour{\emph}
\newcommand\book{\emph}
\newcommand\inbook{\emph}
\def\no#1#2,{\unskip#2, no. #1,} 
\newcommand\toappear{\unskip, to appear}

\newcommand\arxiv[1]{\texttt{arXiv:#1}}
\newcommand\arXiv{\arxiv}
\newcommand\xand{\& }

\end{document}